\documentclass[12pt,oneside,reqno]{amsart}
\usepackage{graphicx,amssymb,amsfonts,amsmath,amsthm,newlfont}
\usepackage[dvipsnames]{xcolor}
\usepackage{mathrsfs} 
\usepackage{bbm} 
\usepackage{mathtools} 
\usepackage{tikz-cd}
\usepackage{enumerate}
\usepackage[tracking]{microtype}
\usepackage{booktabs}

\usepackage[colorlinks=true,linkcolor=red,citecolor=blue]{hyperref}
\usepackage{orcidlink}

\usepackage[backend=bibtex,style=alphabetic,maxnames=9]{biblatex}
\renewbibmacro{in:}{}
\newbibmacro{string+doiurlisbn}[1]{%
  \iffieldundef{doi}{%
    \iffieldundef{url}{%
      \iffieldundef{isbn}{%
        \iffieldundef{issn}{%
          \iffieldundef{eprint}{%
            #1%
          }{%
            \href{https://arxiv.org/abs/\thefield{eprint}}{#1}%
          }
        }{%
          \href{https://books.google.com/books?vid=ISSN\thefield{issn}}{#1}%
      }
      }{%
        \href{https://books.google.com/books?vid=ISBN\thefield{isbn}}{#1}%
      }%
    }{%
      \href{\thefield{url}}{#1}%
    }%
  }{%
    \href{https://doi.org/\thefield{doi}}{#1}%
  }%
}
\DeclareFieldFormat{title}{\usebibmacro{string+doiurlisbn}{\mkbibemph{#1}}}
\DeclareFieldFormat[article,incollection,inproceedings]{title}{\usebibmacro{string+doiurlisbn}{\mkbibquote{#1}}}
\DeclareFieldFormat{postnote}{#1}
\DeclareFieldFormat{eprint}{\href{https://www.arxiv.org/abs/#1}{arXiv:#1}}

\theoremstyle{plain}
\newtheorem{thm}{Theorem}
\newtheorem{lem}[thm]{Lemma}
\newtheorem{cor}[thm]{Corollary}

\newtheorem{remark}[thm]{Remark}
\newtheorem{defn}[thm]{Definition}
\newtheorem{ex}[thm]{Example}

\newcommand{\ZZ}{\mathbbm{Z}}
\newcommand{\QQ}{\mathbbm{Q}}
\newcommand{\RR}{\mathbbm{R}}
\newcommand{\CC}{\mathbbm{C}}
\newcommand{\To}{\longrightarrow}
\newcommand{\hooklongrightarrow}{\lhook\joinrel\longrightarrow}
\DeclareMathOperator{\tr}{tr}
\DeclareMathOperator{\Out}{Out}
\DeclareMathOperator{\Hom}{Hom}
\DeclareMathOperator{\Aut}{Aut}
\DeclareMathOperator{\Ext}{Ext}
\newcommand{\Transpose}{\intercal}

\newcommand{\Graph}[2][1.0]{\vcenter{\hbox{\includegraphics[scale=#1]{fig/#2}}}}

\newcommand{\can}{\mathrm{can}}
\newcommand{\RW}{\mathrm{RW}}
\newcommand{\gm}{\mathfrak{g}^\mm}
\newcommand{\gmr}{\gm_{\RR}}
\newcommand{\hGC}{\mathcal{GC}}
\newcommand{\cGC}{\mathsf{GC}}
\newcommand{\mm}{\mathfrak{m}}
\newcommand{\cochain}{\mathcal{C}}
\newcommand{\Lie}{\mathbb{L}}
\newcommand{\dR}{\mathrm{dR}}
\newcommand{\tone}{\overset{\rightarrow}{1}\!}
\newcommand{\opi}{{}_0 \Pi_{1}}
\newcommand{\MT}{\mathcal{MT}}
\newcommand{\GL}{\mathrm{GL}}
\newcommand{\SL}{\mathrm{SL}}
\newcommand{\Sp}{\mathrm{Sp}}
\newcommand{\sv}{\mathrm{sv}}
\newcommand{\trop}{\mathrm{trop}}
\newcommand{\red}{\mathrm{red}}
\DeclareMathOperator{\gr}{\mathrm{gr}}
\DeclareMathOperator{\Prim}{Prim}
\newcommand{\x}{\mathsf{x}}
\newcommand{\e}{\mathsf{e}}
\newcommand{\comp}{\mathrm{comp}}
\newcommand{\td}{\mathrm{d}} 
\newcommand{\PO}{\mathscr{P}} 
\newcommand{\SPD}{\mathcal{P}} 
\newcommand{\grt}{\mathfrak{grt}}
\newcommand{\ad}{\mathrm{ad}}
\newcommand{\dchB}{{}_0^{}1^B_1}
\newcommand{\dchDR}{{}^{}_01^{\dR}_1}
\newcommand{\defas}{\coloneqq}
\newcommand{\iu}{\mathrm{i}} 
\newcommand{\ipi}{\pi\iu}
\newcommand{\llangle}{\langle\!\langle}
\newcommand{\rrangle}{\rangle\!\rangle}
\newcommand{\kontsevint}{\href{https://bitbucket.org/PanzerErik/kontsevint/}{\texttt{kontsevint}}}

\author{Francis Brown \orcidlink{0000-0002-9295-2572}}
\address{All Souls College, Oxford, OX1 4AL, UK}
\email{francis.brown@all-souls.ox.ac.uk}
\author{Erik Panzer \orcidlink{0000-0002-9897-5812}}
\address{Mathematical Institute, University of Oxford, OX2 6GG, UK}
\email{erik.panzer@maths.ox.ac.uk}
\author{Jean-Luc Portner \orcidlink{0009-0008-3143-9127}}
\address{Mathematical Institute, University of Oxford, OX2 6GG, UK}
\email{jean-luc.portner@maths.ox.ac.uk}

\date{\today}
\title[The motivic Lie algebra embeds into the cohomology of $\GL(\ZZ)$]{The motivic Lie algebra embeds into the cohomology of the general linear group}

\begin{document}

\begin{abstract}
We show that the motivic Lie algebra of mixed Tate motives over $\ZZ$ embeds canonically into the unstable compactly-supported cohomology of locally symmetric spaces for $\GL_g(\ZZ)$, and into the weight-zero compactly-supported cohomology of $\mathcal{A}_g$, the moduli space of principally polarized abelian varieties. Our construction passes through tropical geometry and graph complexes: compactly-supported analogues of the Borel classes pull back via the tropical Torelli map to canonical graph cocycles. The latter were recently identified with cocycles studied previously by Rossi and Willwacher. We combine their results with the theory of single-valued periods to conclude that the cocycles map to generators of the motivic Lie algebra. We also compute the canonical graph cocycles for all graphs with $\leq14$ edges.\vspace{-4mm}
\end{abstract}
\maketitle

\section{Introduction}

The motivic Lie algebra of the title is the graded Lie algebra $\gm$ associated to the category of mixed Tate motives over $\ZZ$. It is a free graded Lie algebra with one generator $\sigma_{2k+1}$ in every odd degree $2k+1\geq 3$. The only known explicit realization of this category is constructed from Deligne’s motivic fundamental groupoid of $\mathbb{P}^1\setminus \{0,1,\infty\}$. Nevertheless, the generators $\sigma_{2k+1}$ originate in algebraic $K$-theory and hence from the stable cohomology of $\GL_g(\ZZ)$. An enduring mystery has been how to reconcile the projective line minus three points with the group $\GL_g(\ZZ)$, which lie in distant parts of the mathematical landscape.

Building on the work of many authors, and especially  refining  recent results in \cite{BCGP}, we show in this paper that the motivic Lie algebra embeds canonically into the compactly-supported cohomology of the locally symmetric spaces associated to $\GL_g(\ZZ)$, 

\begin{thm} \label{intro: thm1} Let $\SPD_g$ denote the space of positive definite symmetric matrices of rank $g$. There is a canonical embedding of the free Lie algebra on generators $\omega_c^{4k+1}$,
\begin{equation} \label{intro:mapomegactoHg}%
    \Lie(\omega^5_c,\omega^9_c,\ldots )
    \hooklongrightarrow
    \bigoplus_{g>0}  H_c^{2g}\big(\SPD_g/\GL_g(\ZZ); \RR\big)
\end{equation}
which maps the generators to classes $\big[\omega^{4k+1}_c\big] \in H_c^{4k+2}(\SPD_{2k+1}/\GL_{2k+1}(\ZZ);\RR)$ that are compactly-supported analogues, shifted by one degree, of the Borel classes (defined below).
\end{thm}

The present result exhibits the motivic Lie algebra directly inside unstable compactly-supported cohomology using natural geometric constructions.
As a consequence, one expects the motivic Lie algebra to appear in a variety of previously unexpected settings, including the cohomology of other arithmetic groups such as the symplectic group. Indeed, the  symplectic analogue is provided by the  moduli space of principally polarised abelian varieties $\mathcal{A}_g$, which may be identified with the  locally symmetric space for $\Sp_{2g}(\ZZ)$. Its  compactly-supported cohomology carries a natural mixed Hodge structure \cite{HodgeII}.

\begin{thm} \label{intro: thm2}
There is a canonical embedding of the same free Lie algebra
\begin{equation*}
    \Lie(\omega^5_c,\omega^9_c,\ldots )
    \hooklongrightarrow
    \bigoplus_{g>0} W_0 H_c^{2g}(\mathcal{A}_g; \RR) 
    .
\end{equation*}
\end{thm}
In Theorems~\ref{intro: thm1} and \ref{intro: thm2}, the target space carries a Hopf algebra structure \cite{BCGP,AMP}. In both cases, the image of the free graded Lie algebra in $\omega_c^{4k+1}$ lands in the subspace of primitive elements and the map is a morphism of graded Lie algebras.

Theorem~\ref{intro: thm2} answers one half of Question~1.17 in \cite{BCGP}; the other half asks if the map in the theorem is surjective onto the subspace of primitive elements (it is for $g<8$). If so, then the motivic Lie algebra is precisely the diagonal part of the primitives in the graded Hopf algebra associated to $\mathcal{A}_g$. 

The proof draws together several previously developed ingredients, including compactly-supported Borel classes, the Hopf algebra structure on the cohomology of the general linear group, and the tropical Torelli map \cite{BMV}. It passes through the moduli spaces of tropical curves $\mathcal{M}_g^{\trop}$, abelian varieties $\mathcal{A}_g^{\trop}$, and the graph complex $\cGC_2$. In particular, each compactly supported Borel class $[\omega_c^{4k+1}]$ pulls back to a graph cohomology class that is represented by a canonical cocycle $\cochain^\can_{4k+1}\in\cGC_2\otimes\RR$ obtained from the canonical integrals $G\mapsto I^\can_G$ of graphs introduced in \cite{BrInvariant}.

The principal new contribution is to apply recent work of the third author \cite{Portner:RWeqCan} which identifies $I^{\can}_G$ with integrals $I^\RW_G$ studied by Rossi and Willwacher in \cite{RossiWillwacher:Etingof}. We relate the latter integrals to the theory of single-valued periods and show that the graph cocycles are motivic (theorem~\ref{thm: maintechnical}). In particular, they thus generate a free Lie subalgebra.
\begin{thm}\label{intro: thm3}
There is a canonical embedding of the free Lie algebra
\begin{equation*}
    \Lie(\omega^5_c,\omega^9_c,\ldots )
    \hooklongrightarrow
    H^0(\cGC_2)\otimes\RR
\end{equation*}
which maps $\omega^{4k+1}_c$ to $[\cochain^\can_{4k+1}]$. More precisely, these classes are motivic: their images in the Ihara Lie algebra $(\Lie(\e_0,\e_1),\{  \ , \  \})\otimes\RR$ lie in, and generate, the motivic Lie subalgebra $\gm\otimes \RR$. Therefore, we obtain a canonical isomorphism
\begin{equation*}
    \Lie(\omega^5_c,\omega^9_c,\ldots ) \otimes\RR
    \cong
    \gm\otimes\RR.
\end{equation*}
\end{thm}    
In this theorem we can replace $\RR$ with the ring of single-valued multiple zeta values.

Theorem~\ref{intro: thm3} circumvents unresolved conjectures on the Grothendieck–Teichm\"uller Lie algebra $\grt_1$ that underpinned partial results of \cite{BCGP} and relied on computer calculations in low degrees. A very recent result on $\grt_1$ by Willwacher \cite[Theorem~1.2]{Willwacher:freeGRT} can be combined with the non-vanishing of the canonical integrals of wheel graphs \cite{BrownSchnetz:WheelCan} to obtain an alternative proof (avoiding any motivic machinery) of the freeness of the canonical graph cocycles (the first part of theorem~\ref{intro: thm3})---and therefore also of theorems \ref{intro: thm1} and \ref{intro: thm2}. In our approach, the role of $\grt_1$ is replaced by $\gm$ and we derive from \cite[Theorem~1.4]{RossiWillwacher:Etingof} and \cite{Portner:RWeqCan} the stronger result (second part of theorem~\ref{intro: thm3}) that establishes the canonical cocycles directly inside the motivic Lie subalgebra $\gm\subseteq \grt_1$.

This insight gives a new perspective on the landmark result of \cite{CGP}, which identifies a copy of $\gm$ inside the cohomology of $\mathcal{M}_g$, the moduli space of stable curves of genus $g$. That result relied on Willwacher's identification $H^0(\cGC_2)\cong\grt_1$. Instead of this, our approach via theorem~\ref{intro: thm3} directly implies that $\gm\otimes\RR\subseteq H^0(\cGC_2)\otimes\RR$, without any reference to $\grt_1$. At present the precise relationship between the result of \cite{CGP} and theorem \ref{intro: thm2} is unclear. Another significant difference in the present setting is the existence of the Hopf algebra structures on the cohomology of $\GL_g(\ZZ)$ and $\mathcal{A}_g$ which generate large families of additional cohomology classes from the $\omega_c^{4k+1}$ (see, for example, Corollary~\ref{cor: tensoralg}).

We now describe the principal mathematical objects which are involved. 

\subsection{Overview}
Let $\mathcal{M}^{\trop}_g$ denote the moduli space of tropical curves, or equivalently, of weighted stable metric graphs of genus $g$. Let $\mathcal{A}^{\trop}_g$ denote the moduli space of tropical abelian varieties, which may be described as lattices equipped with a positive semi-definite quadratic form. More precisely, it is the quotient $\SPD_g^{rt}/\GL_g(\ZZ)$ of the rational closure of $\SPD_g$, equipped with its Satake topology. It is the Satake compactification of the locally symmetric space $\SPD_g/\GL_g(\ZZ)$. 

To a metric graph $G$ one may associate the graph Laplacian, which defines a quadratic form on $H_1(G;\ZZ)$. It induces the tropical Torelli map  
\[
    \lambda\colon\quad\mathcal M_g^{\trop}\longrightarrow \mathcal A_g^{\trop},
    \qquad \mathcal A_g^{\trop}
    \supset \SPD_g/\GL_g(\ZZ)
    \,.
\]
We begin with the space on the far right. Let $X_{ij}=X_{ji}\colon\SPD_g\longrightarrow \RR$ denote the coordinate functions (matrix entries) on $\SPD_g\subset\RR^{g(g+1)/2}$. The Borel classes in the stable cohomology of $\GL_g(\ZZ)$ are represented by the differential forms 
\begin{equation} \label{intro: canform}
    \omega^{4k+1} = \tr \left( (X^{-1} \td X)^{4k+1}\right)\,,
\end{equation}
on $\SPD_g/\GL_g(\ZZ)$. In this paper the same forms \eqref{intro: canform} play a different, but closely related role.
Indeed, we consider classes $[\omega^{4k+1}_c]$ defined in \cite{BrBord} which are compactly-supported analogues of the classes \eqref{intro: canform} in degree one higher. They are constructed as relative de Rham classes $(\widetilde{\omega}^{4k+1}, 0)$ on the Borel--Serre compactification of $\RR^{\times}_{>0} \backslash \SPD_g/\GL_g(\ZZ)$ relative to its boundary, where $g=2k+1$.
The main subtlety is to show that the forms $\omega^{4k+1}$, which blow up at infinity, extend smoothly to forms $ \widetilde{\omega}^{4k+1}$ on the Borel--Serre compactification, and hence to $\mathcal{A}_g^{\trop}$. The reason for the degree shift $[\omega^{4k+1}_c]\in H_c^{4k+2}$ in theorem~\ref{intro: thm1} is the quotient by $\RR^{\times}_{>0}$, which has non-trivial compactly-supported cohomology.
 
By pulling back the classes $[\omega_c^{4k+1}]$ along the tropical Torelli map, we obtain differential forms on $\mathcal{M}_g^{\trop}$, or its link  $L\mathcal{M}_g^{\trop}$. These spaces are  not manifolds. Nevertheless  such a form defines a map from cells on  $L\mathcal{M}_g^{\trop}$, which are indexed by combinatorial  stable weighted graphs $G$, to numbers. Concretely this map is given by integrating the forms \eqref{intro: canform} over the cell $\sigma_G$ of all metric graphs of a fixed combinatorial type $G$. 
These are called the \emph{canonical} graph integrals $I^{\can}_G$ and are discussed in \S\ref{sect: cochains}. The assignment $G\mapsto I_G^{\can}$ defines a closed real-valued cocycle on the graph complex $\hGC_2$. The latter is naturally identified with the relative cellular homology of $L\mathcal{M}_g^{\trop}$ relative to its boundary.  
Technically, this requires working with suitable bordifications and relative cohomology groups.

The next fundamental step is to identify these canonical graph integrals with the Rossi--Willwacher integrals $I^{\RW}_G$. This uses the recent result \cite{Portner:RWeqCan} that the integrals $I^{\can}_G$ and $I^{\RW}_G$ coincide. The latter are versions of Kontsevich's integrals in deformation quantisation. The former can be thought of as  Feynman integrals in a  parametric space  representation; the latter as a position space representation. The passage between the two crucially uses the Schwinger trick from quantum field theory. 

It follows that the canonical cocycle and the Rossi--Willwacher cocycle on the graph complex are one and the same. To relate graphs to the motivic Lie algebra, we consider a map \cite[\S7]{RossiWillwacher:Etingof} from graphs to Lie words
\begin{equation*}
    \phi\colon H^0(\cGC_2) \To 
    ( \widehat{\Lie}(\e_0,\e_1), \{ , \}) \cong  \mathrm{Lie}\, ( {}_0 \Pi_1  , \circ)
\end{equation*}
where ${}_0 \Pi_1 = \pi_1^{\dR}(\mathbb{P}^1 \setminus \{0,1,\infty\}, \tone_0, -\tone_1)$ is the de Rham fundamental group of $\mathbb{P}^1\setminus \{0,1,\infty\}$ equipped with the Ihara composition law $\circ$. 

The final ingredient is to exploit the relation between the image $\phi_{2k+1}=\phi(\cochain^\can_{4k+1}$) of the graph cocycle and single-valued periods of mixed Tate motives. Combining \cite[Theorem~1.4]{RossiWillwacher:Etingof} and \cite[\S5]{brownSVMZV}, we give a motivic interpretation of the Lie words $\phi_{2k+1}$ via the single-valued Drinfeld associator. 
Indeed, since the latter lies in the Zariski closure of  the unipotent radical $U_{\MT}$ of the motivic Galois group, we can unravel the $\phi_{2k+1}$ using a Magnus expansion inside $U_{\MT}$.  This enables us to establish that the graph cocycles defined by $I^{\can}_G= I^{\RW}_G$ take values in the motivic Lie subalgebra $\gm\subset \mathrm{Lie}\, ( {}_0 \Pi_1  , \circ)$.

\subsection{Summary} 
The preceding constructions define a map of Lie algebras:
\begin{equation*}
\begin{tikzcd}[column sep=6mm,row sep=3mm]
\Prim \Big( \bigoplus\limits_{g>0}  H_c^{2g}(\SPD_g/\GL_g(\ZZ); \RR) \Big) \arrow[r] &
H^0(\cGC_2) \otimes \RR \arrow[r] &
(\widehat{\Lie}(\e_0,\e_1) \otimes \RR, \{ , \} )  \arrow[r,phantom,"\supset"] &[-4mm]
\gm\otimes\RR
\\
\quad \big[\omega_c^{4k+1}\big] \quad
\arrow[u,phantom,sloped,"\in" {xshift=5pt}]
\arrow[r,mapsto]
&
\ \, \big[\cochain_{4k+1}^\can\big] \ \,
\arrow[u,phantom,sloped,"\in"]
\arrow[rr,mapsto]
& &
\ \phi_{2k+1} \ 
\arrow[u,phantom,sloped,"\in"]
\end{tikzcd}
\end{equation*}
Indeed, a result of \cite{BCGP} shows that the first map, induced by the tropical Torelli map $\lambda$, is a morphism of graded Lie algebras; whereas the second is induced by the map $\phi$. The primitives here are taken with respect to the graded Hopf algebra structure on the compactly-supported cohomology of $\GL_g(\ZZ)$ constructed in \cite{BCGP}.

Theorem \ref{intro: thm1} follows by showing that the diagonal arrow  in the commutative diagram
\begin{equation*}
\begin{tikzcd}[column sep=4em, row sep=3em]
\Lie(\omega^5_c,\omega^9_c,\ldots)\otimes_{\QQ}\RR
  \arrow[r]
   \arrow[dr, "\cong"] &
(\Lie(\e_0,\e_1)\otimes_{\QQ}\RR, \{\, , \,\})
\\
&
\gm\otimes_{\QQ} \RR \cong \Lie(\sigma_3,\sigma_5,\ldots)\otimes_{\QQ}\RR
  \arrow[u, hook]
\end{tikzcd}
\end{equation*}
is an isomorphism, where the horizontal arrow is $\eqref{intro:mapomegactoHg}$ composed with $\lambda^*$ and $\phi$.

The paper is structured as follows: In \S\ref{sect: cochains}, we review the graph complex, canonical and Rossi--Willwacher integrals $I^\can_G$ and $I^\RW_G$, and the map $\phi$. In \S\ref{sec:motivics} we gather the motivic ingredients, in particular Deligne's motivic fundamental group $\mathbb{P}^1 \setminus \{0,1,\infty\}$, the Galois group, Ihara's action, and motivic as well as single-valued associators. These are applied in \S\ref{sec:can-are-motivic} to derive from \cite{RossiWillwacher:Etingof} the key technical result (theorem~\ref{thm: maintechnical}) that, canonically, $\gm\otimes\RR\cong\Lie(\phi_3,\phi_5,\ldots)\otimes\RR$. In \S\ref{sec:geometry} we explain the geometric connections that lead from this result to theorems \ref{intro: thm1} and \ref{intro: thm2}.
In appendix~\ref{sec:explicit-cocycles} we present the explicit forms of the canonical graph cocycles $\cochain^\can_5$, $\cochain^\can_9$ (which were known before) and $\cochain^\can_{13}$ (which is new).

\subsection{Further directions and open questions}

\subsubsection{Additional classes, and non compactly-supported cohomology}
By exploiting the Hopf algebra structure and the Milnor--Moore theorem, one obtains many more classes from theorems~\ref{intro: thm1}, \ref{intro: thm2}. For example, we have the following corollary.
\begin{cor} There are canonical embeddings from the bigraded tensor algebra on classes $\omega_c^{4k+1}$, where the degree of $\omega_c^{4k+1}$ is $4k+2$ and its genus is $2k+1$:
\begin{equation*}
    T \left( \bigoplus_{k\geq 1} \QQ\, \omega^{4k+1}_c\right)
    \hooklongrightarrow
    \begin{cases}
        \bigoplus_{g\geq 0}  H_c^{2g}(\SPD_g/\GL_g(\ZZ); \RR)  \\ 
        \bigoplus_{g\geq 0} W_0 H_c^{2g}(\mathcal{A}_g; \RR)
    \end{cases}
\end{equation*}
where $\mathcal{A}_0$ and $\SPD_0$ are defined to be a point and $\GL_0(\ZZ)\defas\{1\}$. 
\end{cor} 
By Poincar\'e duality, a copy of these cohomology classes, and in particular the motivic Lie algebra, also appears in the homology of the special linear groups $\SL_g(\ZZ)$. This mechanism is well-documented, for example in \cite{BCGP}.
These classes are expected to be of Eisenstein type, and therefore should be contrasted with the cuspidal classes recently discovered in \cite{BoxerCalegariGee:CuspGLn}.

\subsubsection{Stable versus unstable classes}
In order to close the gap between algebraic $K$-theory and the compactly-supported cohomology classes $[\omega_c^{4k+1}]$, one may invoke the Quillen spectral sequence. Its $E_1$ page is the Hopf algebra $\mathcal{H}_c$ (of theorem~\ref{thm:GL-Hopf}) of \emph{unstable} cohomology groups of $\GL_g(\ZZ)$ with compact support, and it abuts to the \emph{stable} cohomology of $\GL_g(\ZZ)$. See \cite{BCGP} for a detailed discussion. 
One expects that the classes $[\omega^{4k+1}_c]$ on the $E_1$ page are permanent, and that the corresponding period integrals, which are given by the canonical integrals over the wheel classes \cite{BrownSchnetz:WheelCan}, compute the Borel regulator, which is proportional to odd zeta values by Borel. See \cite[\S2.4]{BrownSchnetz:WheelCan} for  an alternative mechanism for relating the wheel integrals to the Borel regulator in low degrees.

\subsubsection{Double copy and doubling of the weight}
In the theory of mixed Tate motives, the generators $\sigma_{2k+1}$ of $\gm$ in degree $2k+1$ arise from the algebraic $K$-theory in roughly double the degree:
the image $[\sigma_{2k+1}]$ in the abelianisation of $\gm$ generates the extension group
\begin{equation*} 
    \Ext_{\MT(\ZZ)}^1(\QQ(0), \QQ(2k+1)) 
    \cong   K_{4k+1}(\ZZ)\otimes_{\ZZ} \QQ 
     \cong \Prim \, \Big( \varprojlim_g H^{4k+1}(\GL_g(\ZZ);\QQ)\Big)
\end{equation*}
whereas the stable cohomology on the right is generated (over $\RR$) by the Borel class $[\omega^{4k+1}]$. On the left, the odd zeta value $\zeta(2k+1)$ dual to $\sigma_{2k+1}$ can be realized as an iterated integral of length $2k+1$ in the motivic fundamental group of $\mathbb{P}^1\setminus \{0,1,\infty\}$. On the right, the same zeta value arises as an integral of $\omega^{4k+1}$,  for some large $g\gg k$ see \cite{Borel}, and, granting the discussion in the previous  paragraph, should coincide with the ordinary integral of  $\omega^{4k+1}$ over the  wheel  homology classes \cite{BrownSchnetz:WheelCan} in  the unstable range.

This puzzling halving of the weight can be explained as follows.
On the right, $I^\can_G$ is the integral of the holomorphic form $\omega^{4k+1}$ in $4k+2$ projective edge length variables $x_e\in\RR$ (coordinates on $\mathcal{M}_g^\trop$) and vanishes unless $G$ has $2k+2$ vertices and $g=2k+1$. In contrast, Rossi--Willwacher (definition~\ref{def:RW}) integrate a degree $4k$ form in holomorphic and antiholomorphic variables $z_v,\bar{z}_v$ over the configuration space $\mathcal{M}_{0,2k+3}(\CC)$ of complex vertex positions $z_v\in\CC$. The identity $I^\can_G=I^\RW_G$, proved in \cite{Portner:RWeqCan} via a correspondence, thus transports canonical integrals from $\mathcal{M}_{2k+1}^\trop$ to the moduli space $\mathcal{M}_{0,2k+3}(\CC)$.\footnote{This correspondence between integrals of $\omega^{4k+1}$ on $\SPD_g/\GL_g(\ZZ)$ and real-analytic forms on complements of hyperplane arrangements generalizes from graphs to arbitrary Voronoi cells \cite[\S2]{Portner:RWeqCan}.}

Now in $I^\RW_G$, each complex integral over $\td z_v\wedge \td\overline{z}_v$ produces a factor $2\ipi$ by the residue formula  and reduces to a single real integration \cite{schnetzgraphical,BanksPanzerPym:MZVinDQ}. This explains the halving of the weight. It would also follow from expanding $I^\RW_G$ into so-called \emph{single-valued integrals}
\begin{equation*}
    \frac{1}{(2\ipi)^{2k+1}}
    \int_{\mathcal{M}_{0,2k+4}(\CC)} \eta \wedge \overline{\nu}
    =\sv\left(\int_{\gamma} \nu \right)
\end{equation*}
of holomorphic and antiholomorphic logarithmic forms $\eta$ and $\overline{\nu}$. The `double copy formula' from \cite{BrownDupont:SVdoublecopy} expresses such integrals as polynomials of weight $2k+1$ in ordinary (holomorphic), relative ($\gamma$ has boundary) periods of $\mathcal{M}_{0,2k+4}$. The latter reduce to iterated integrals on $\mathcal{M}_{0,4}=\mathbb{P}^1\setminus \{0,1,\infty\}$ (see \cite{Brown:MZVPeriodsModuliSpaces}) and thus bring us to the standard realization of mixed Tate motives over $\ZZ$. 
The integrand of $I^\RW_G$, however,  does not immediately  fit into the single-valued framework, so realizing this approach requires further (ongoing) work.

To conclude, the Schwinger trick correspondence \cite{Portner:RWeqCan} between canonical integrals on $\mathcal{M}_g^\trop$ and (single-valued) integrals $I^\RW_G$ on the moduli space $\mathcal{M}_{0,n}(\CC)$ of genus 0 curves with $n$-marked points provides the missing link to connect the Borel forms $\omega^{4k+1}$ on $\SPD_g$ with the motivic fundamental group of $\mathbb{P}^1\setminus \{0,1,\infty\}$.

\subsection{Acknowledgments}
This project has received funding from the European Research Council (ERC) under the European Union’s Horizon Europe programme (grant agreement No.\ 101167287).
Erik Panzer is funded as a Royal Society University Research Fellow through grant {URF{\textbackslash}R{\textbackslash}251041}.

For the purpose of Open Access, the authors have applied a CC BY public copyright licence to any Author Accepted Manuscript (AAM) version arising from this submission.


\section{Cocycles on the graph complex} \label{sect: cochains}

\subsection{The even commutative graph complex}
Let $G$ be a finite connected graph with edge set $E_G$ of size $e_G = |E_G|$ and vertex set $V_G$. An \emph{orientation} on $G$ is an element 
\[
    o \in \big(\textstyle{\bigwedge}^{e_G}\, \ZZ^{E_G} \big)^{\times}
    \cong \ZZ^{\times}=\{-1,+1\}
    \ .
\]
The graph homology complex $\hGC_2$ is the $\QQ$-vector space generated by classes $[G,o]$ of connected oriented graphs that are simple (no multiedges or self-loops) and have degree $\geq 3$ at every vertex, subject to the relations $[G,-o]=-[G,o]$ and $[G,o] = [G', \sigma(o)]$ for any isomorphism $\sigma\colon G \overset{\sim}{\rightarrow} G'$.
The differential satisfies $\partial^2=0$ and is defined by 
\[ \partial [G,o] = \sum_e  [G/e, o/e] \]
where $G/e$ denotes the graph obtained from $G$ by contracting the edge $e$ and $o/e$ is the induced orientation such that $e\wedge o/e = o$. Whenever $G/e$ is not simple (e.g.\ $e$ lies in a triangle of $G$), the summand is set to $[G/e,o/e]=0$.
\begin{figure}
    \centering
    $
        G/b\cong\Graph{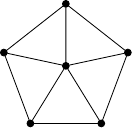}
        \quad\reflectbox{$\mapsto$}\quad
        G=\Graph{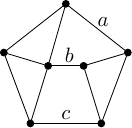} 
        \quad\mapsto\quad 
        G/a\cong G/c\cong \Graph{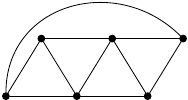}
    $
    \caption{The graph complex differential (edge contractions).}%
    \label{fig:boundary}%
\end{figure}
\begin{ex}\label{ex:boundary}
    The graph $G$ in figure~\ref{fig:boundary} admits three non-vanishing edge contractions. One produces a wheel graph $W_5$ and two generate a zigzag graph $Z_5$. With their orientations specified in table~\ref{tab:can357}, and a suitable orientation of $G$, this gives $\partial G=W_5+2Z_5$.    
\end{ex}
The differential preserves the grading $\hGC_2=\bigoplus_h \gr_h \hGC_2$ by the first Betti number $h_G = e_G -|V_G|+1$ and has degree $-1$ with respect to the homological degree, which is defined as $e_G-2h_G$. The graph homology
\[  
    H(\hGC_2) 
    = \frac{\ker \partial}{\mathrm{Im} \, \partial}
    =\bigoplus_{n \geq 0 } H_n(\hGC_2)
    =\bigoplus_{n,h \geq 0 } \gr_h H_n(\hGC_2)
\]
is concentrated in non-negative degrees by \cite[Theorem~1.1]{WillwacherGRT} or \cite[Theorem~1.4]{CGP}. The spaces $\gr_h\hGC_2$ and $\gr_h H_n(\hGC_2)=H_n(\gr_h \hGC_2)$ have finite dimension for all $n,h$.

The graph cohomology complex is the product $\cGC_2=\prod_h \gr_h \hGC_2$ endowed with a codifferential $\delta\colon \cGC_2\To\cGC_2$ defined by splitting vertices, see \cite[\S3]{WillwacherGRT}. This complex is the dual $\cGC_2\cong \Hom(\hGC, \QQ)$ of $\hGC_2$ via the pairing $\langle\cdot,\cdot\rangle\colon \cGC_2\otimes \hGC_2\longrightarrow \QQ$ given by
\begin{equation*}
    \big\langle[G,o],[G',o']\big\rangle
    =
    \begin{cases}
        \phantom{-}|\Aut G| & \text{if $[G,o]=[G',o']\neq 0$,} \\
                  -|\Aut G| & \text{if $[G,o]=-[G',o']\neq 0$,} \\
               \phantom{-}0 & \text{if $G$ is not isomorphic to $G'$.} \\
    \end{cases}
\end{equation*}
With this identification, the codifferential is the transpose $\langle \delta G,G'\rangle = \langle G,\partial G'\rangle$ of the differential, see \cite[\S4.3]{BrownHuPanzer:Pfaffian}. A cocycle, i.e.\ a linear map $f\colon \hGC_2\longrightarrow \QQ$ with $f\circ \partial =0$, thus corresponds to an element $x\in\cGC_2$ with $\delta x=0$, via the relation $f(G)=\langle x,G\rangle$.

Furthermore, the complex $\cGC_2$ has a Lie algebra structure defined by inserting graphs into each other, see\ \cite[Example~4.8]{BrownHuPanzer:Pfaffian}. Under the identification $\cGC_2\cong\Hom(\hGC_2,\QQ)$ from above, this Lie bracket is given for two linear maps $f_1,f_2\colon \hGC_2 \longrightarrow \QQ$ by a sum
\[
    [f_1,f_2] (G)  
    =  \sum_{\gamma}  \big(f_1(\gamma) f_2(G/\gamma) - f_2(\gamma) f_1(G/\gamma) \big)
\]
over all connected subgraphs $\gamma \subset G$ such that both $\gamma$ and the quotient graph $G/\gamma$ are simple and have degrees $\geq 3$ at each vertex; see \cite[Theorem~4.14]{BrownHuPanzer:Pfaffian}.

\subsection{The canonical cocycle}
Choose a direction for each edge, to identify $\ZZ^{E_G}$ with the 1-chains of a graph viewed as a simplicial complex. Consider the exact sequence
\[ 0 \To H_1(G;\ZZ) \To \ZZ^{E_G} \To \ZZ^{V_G} \ .\]
Let $\QQ[x_e, e\in E_G]$ denote the polynomial ring with one generator $x_e$ for each edge $e$ of $G$. Consider the quadratic form on $\ZZ^{E_G}$ defined by $\langle e_i, e_j\rangle = \delta_{ij} x_i$, where $\delta$ denotes the Kronecker delta.  It restricts to a quadratic form on $H_1(G;\ZZ)$, called the graph Laplacian. In any choice of basis of $H_1(G;\ZZ)$, it is represented by an $h_G\times h_G$ square matrix with entries in $\QQ[x_e]$ known as a graph Laplacian matrix $\Lambda_G$. 

For a connected graph $G$ as above, consider the coordinate simplex in projective space
\begin{equation}\label{eq:graph-simplex}
    \sigma_G = \{[x]\colon x_e \geq 0 \ \text{for all} \ e\in E_G \} \subset \mathbb{P}^{e_G-1} (\RR)\ .
\end{equation}
From an orientation on $G$ we induce an orientation on $\sigma_G$ as follows: label the $n=e_G$ edges such that $o=e_1\wedge\ldots\wedge e_n$, then declare $\partial_{e_2}\wedge\ldots\wedge\partial_{e_n}$ positively oriented in the affine chart $x_{e_1}=1$. The following integrals were defined in \cite{BrInvariant}, denoted there $I_G(\omega^{4k+1})$.
\begin{defn}\label{defn:Ican}
For any oriented graph with $e_G=4k+2$ edges (with $k\geq 1$ any integer), set
\begin{equation*}
    I^\can_G \defas \int_{\sigma_{G}} \tr \big( (\Lambda^{-1}_G \td\Lambda_G )^{4k+1} \big)\ . 
\end{equation*}
\end{defn}
Since the determinant of the Laplacian $\Lambda_G$ is the Kirchhoff graph polynomial, one may interpret the integrals $I^\can_G$ as instances of Feynman periods \cite{BlochEsnaultKreimer:MotivesGraphPolynomials}. 

We briefly summarise some facts established in \cite{BrInvariant}:

\begin{enumerate}[I.]
    \item The integrand does not depend on the choice of edge directions or basis of $H_1(G;\ZZ)$. It defines a differential form of degree $e_G-1$ in the edge variables $x_e$, which descends to a meromorphic differential form on the projective space $\mathbb{P}^{e_G-1}$.

    \item \label{II} The integrand has poles along the boundary of the domain of integration $\sigma_G$, but still the integral is finite. It only depends on the class $[G,o]$ in the graph complex $\hGC_2$.

    \item \label{III} The integral $I^\can_G$ vanishes unless $e_G=2 h_G$ (\cite[Lemma~8.4]{BrInvariant}). This is equivalent to $h_G=2k+1$, and so $[G,o]$ has homological degree zero in $\hGC_2$.

    \item \label{IV} The integrals $I^\can_G$ of different graphs are related: for any graph $G'$ with $4k+3$ edges,
    \[ 0 = \sum_{e \in G'} I^\can_{G'/e} \ . \]
    This is a special case of an application of Stokes' formula, because the additional terms in \cite[Theorem~8.5]{BrInvariant} vanish in this case (firstly because $\omega^{4k+1}$ is primitive in the terminology of \emph{loc.\ cit.}, and secondly by item \ref{III}).
\end{enumerate}
By item \ref{II}, the integral extends to a linear map $I^\can\colon \hGC_2\longrightarrow\RR$. By \ref{IV}, this satisfies $I^\can(\partial G')=0$, so this is a cocycle of the complex $(\hGC_2,\partial)$. In graph cohomology, this function $I^\can=\langle\cochain^\can, \cdot\rangle$ corresponds to a graph cocycle which we denote $\cochain^{\can}\in\cGC_2\widehat{\otimes}\RR$. Each component $\cochain^\can_{4k+1}\in\cGC_2\otimes\RR$ of $\cochain^\can=\prod_k\cochain^\can_{4k+1}$ is a sum over the finitely many isomorphism classes of graphs with $4k+2$ edges,
\begin{equation}\label{eq:graph-cocycle-def}
\cochain^{\can}_{4k+1} = \sum_{G} \frac{1}{\Aut(G)} I^\can_{[G,o]} \cdot [G,o].
\end{equation}
By \ref{III}, only graphs with $h_G=2k+1$ contribute, so $\cochain^\can_{4k+1}\in\gr_{2k+1} \cGC_2\otimes\RR$. The cohomology classes represented by these cocycles are thus of degree zero: $[\cochain^{\can}_{4k+1}] \in H^0(\cGC_2) \otimes \RR$.
\begin{cor} For all  $k\geq 1$, the cochain  $\cochain^{\can}_{4k+1}$ is a cocycle: $\delta  \cochain^{\can}_{4k+1}=0$.
\end{cor}
\begin{ex}
    The canonical integrals of the wheel and zigzag graphs from figure~\ref{fig:boundary} were computed in \cite[\S10.2]{BrInvariant}. In our conventions and with the orientations specified in table~\ref{tab:can357}, these are $I^\can_{W_5}=1260\zeta(5)$ and $I^\can_{Z_5}=-630\zeta(5)$. So indeed, they satisfy the cocycle constraint $I^\can(W_5+2Z_5)=0$ from example~\ref{ex:boundary}. The corresponding graph cocycle is
    \begin{equation*}
        \cochain^\can_9 = 126\zeta(5)\cdot W_5-315\zeta(5)\cdot Z_5
        =126\zeta(5)\left( \Graph[0.4]{W5}-\tfrac{5}{2} \times\Graph[0.5]{Z5} \right).
    \end{equation*}
\end{ex}

\subsection{The Rossi--Willwacher cocycle}
Let $k \geq 1 $ and let $[G,o]$ be an oriented graph with $4k+2$ edges and $2k+2$ vertices as above.
Denote coordinates on $\CC^{V_G}$ by $z_v\in \CC$, for every vertex $v\in V_G$.
The configuration space of vertex positions is
\[  C(V_G) =  (\CC^{V_G} \setminus \Delta) /  ( \CC \rtimes \RR_{>0}^{\times})  \]
where $\Delta = \{z \in \CC^{V_G} \mid z_u = z_v \ \text{for some}\ u \neq v\}$ denotes the large diagonal, and where $\CC \rtimes \RR_{>0}^{\times}$ acts diagonally on $\CC^{V_G}\setminus \Delta$ by translation and real positive scaling.

Choose any ordering of the edge set that is compatible with the orientation, i.e.\ so that $o=e_1\wedge\ldots\wedge e_{4k+2}$.
Direct each edge and set $z_i \defas z_{s(i)}-z_{t(i)}$ where $s(i),t(i) \in V_G$ are the source and target of $e_i$.
The following integrals were defined in \cite[eqs.\ (9), (16)]{RossiWillwacher:Etingof}.\footnote{%
To simplify the comparison with $I^\can_G$, our normalization of the integrals $I^{\RW}_G$ relates to the original definition of $c^{1/2}_G$ in \cite[eq.~(16)]{RossiWillwacher:Etingof} by $I^\RW_G=(2 \ipi)^{2k+1}\, 2^{4k}c_G^{1/2}$.}
\begin{defn} \label{def:RW}
For an oriented graph $[G,o]$ as above, set
\begin{equation*}
    I^{\RW}_G \defas
    \frac{(-1)^k 2^{2k-1} }{\pi^{2k+1}}
    \int_{C(V_G)} \sum^{4k+2}_{i=1} (-1)^{i-1} \log |z_i|^2 \, \td  \arg(z_1) \wedge \ldots \wedge  \widehat{\td \arg(z_{i})}\wedge \ldots \wedge \td\arg(z_{4k+2})
    .
\end{equation*}
\end{defn}
Note that $C(V_G)\cong S^1\times\mathcal{M}_{0,2k+3}(\CC)$, so after a trivial integral over global rotations, $I^{\RW}_G$ amounts to the integral of a real-analytic form (a logarithm and angles) over $\mathcal{M}_{0,2k+3}(\CC)$.

We briefly summarise some facts established in \cite{RossiWillwacher:Etingof}:
\begin{enumerate}[I.]
    \item The integrand is singular on the diagonal $\Delta$, but still the integral is finite. It only depends on the class $[G,o]$ in the graph complex $\hGC_2$.
    \item For any graph $G'$ with $4k+3$ edges and $2k+3$ vertices, the relation $0=\sum_{e\in G'} I^\RW_{G'/e}$ holds, see \cite[Proposition~6.1]{RossiWillwacher:Etingof}. This follows from a regularized version of Stokes' theorem for forms with logarithmic singularities \cite{AlekseevRossiTorossianWillwacher:Logs}.
\end{enumerate}
Therefore, the Rossi-Willwacher integral also extends to a cocycle $I^\RW\colon \hGC_2\longrightarrow \RR$. The corresponding graph cohomology cocycle satisfies $\delta\cochain^\RW=0$ and its graded pieces $\cochain^\RW_{4k+1}\in\cGC_2\otimes\RR$ consist of graphs with $4k+2$ edges and $2k+2$ vertices.

\subsection{Equality of the cocycles}
In \cite{Portner:RWeqCan} the third-named author proved
\begin{thm}\label{thm:RWeqCan}
The canonical and Rossi--Willwacher integrals are equal, $I^\can_G=I^\RW_G$, for all graphs. Therefore, the corresponding graph cocycles are identical:
\[  \cochain^{\can} = \cochain^{\RW} \ . \]
\end{thm}

\subsection{From graphs to words} \label{sect: phi} 
Willwacher constructed in \cite{WillwacherGRT} a morphism of graded Lie algebras from the degree-zero cohomology of the graph complex to the completed Ihara Lie algebra (see \S\ref{sec:Ihara-torsor} for the definition of its Lie bracket):
\begin{equation}\label{eq:graphs-to-words}
    \phi\colon H^0(\cGC_2) \to ( \widehat{\Lie}(\e_0,\e_1), \{  \ , \ \}) \ .
\end{equation}
It is induced by a map, which we also denote by $\phi\colon \gr\cGC_2 \rightarrow \Lie(\e_0,\e_1)$, that is defined on individual graphs. We sketch its definition; see \cite{WillwacherGRT} and \cite[\S7]{RossiWillwacher:Etingof} for details.

Let $[G,o]\in\cGC_2$ be an oriented graph.
For any vertices $x,y\in V_G$ and any edge $f\in E_G$, define a Lie word $\phi_{x,y}^f\in \Lie(\e_0,\e_1)$ as follows:
\begin{itemize}
    \item We require that $x$ and $y$ are adjacent (the edge $xy\in E_G$ exists), that $f$ is incident to $x$ but not to $y$, that $G\setminus\{x,y\}$ is a tree, and that every vertex $v\in V_G\setminus\{x,y\}$ is trivalent. If any of these conditions are violated, set $\phi_{x,y}^f\defas 0$.
    \item If all conditions hold, root the graph $G\setminus\{f,xy\}$ at the vertex $r$ on the other end of $f=xr$. Then each vertex $v\in V\setminus\{x,y\}$ has precisely two children $v',v''\in V$, i.e.\ outgoing edges $vv'$ and $vv''$. Assign Lie polynomials $L\colon V_G\longrightarrow \Lie(\e_0,\e_1)$ to each vertex by $L(x)\defas \e_1$, $L(y)\defas \e_0$, and the recursion
    \begin{equation*}
        L(v)=
            [L(v'),L(v'')] .
    \end{equation*}
    \item Define on orientation of $G$ by
    \begin{equation*}
        o_{x,y}^f\defas f\wedge xy \bigwedge_{v\in V_G\setminus\{x,y\}} vv'\wedge vv''.
    \end{equation*}
    \item Set $\phi_{x,y}^f\defas L(r)\cdot o^f_{x,y}/o$.
\end{itemize}




Note that the sign $o_{x,y}^f/o\in\{1,-1\}$ and $L(r)$ individually depend on the choice of child labels $v'\leftrightarrow v''$, but their product $\phi_{x,y}^f(a,b)$ does not.
We can thus define
\[
\phi(G) = \sum_{x,y,f}
\left(\phi_{x,y}^{f} - \phi_{x,y}^f\big|_{\e_0\leftrightarrow\e_1}\right)
.\] 
By construction, $\phi(G)$ is antisymmetric under the interchange of $\e_0$ and $\e_1$.
Also note that the length of $\phi(G)$ is equal to $n$, the number of vertices of $G$ minus 1.

\begin{figure}
    \centering
    \includegraphics{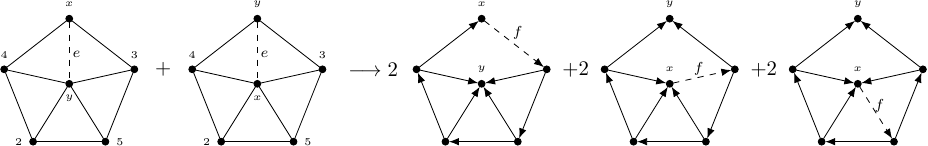}
    \caption{The 3 types of Lie trees obtained from the wheel with 5 spokes graph.}%
    \label{fig:phi-trees}%
\end{figure}
\begin{ex}
	Consider the five wheel graph $W_5$. For a non-zero contribution $x$ or $y$ must be the $5$-valent central vertex.
    The five possible choices for the other vertex are isomorphic. 
	Fixing such a pair $(x,y)$, there are two choices for $f$ if $x$ is a trivalent vertex and four choices if $x$ is the center vertex.
    Identifying terms by reflection in the vertical axis, we obtain the three directed trees of figure~\ref{fig:phi-trees}.
    
    Labelling $x$ as $\e_0$ and $y$ as $\e_1$ we find the following three Lie elements respectively, where the signs are determined by orienting $W_5$ with $o = 
    xy \wedge x2 \wedge x3 \wedge x4 \wedge x5 \wedge 3y \wedge 35 \wedge 4y \wedge 42 \wedge 52$:
    \[
    2 [[[[\e_0, \e_1], \e_1], \e_1], \e_1] + 2 [\e_1, [[[\e_1, \e_0], \e_0], \e_0]] - 2 [[\e_1, \e_0], [[\e_1, \e_0], \e_0]]
    \]
	Considering the opposite labelling $x\leftrightarrow y$ and accounting for the symmetry factors, we get
	\begin{align*}
\phi(W_5) = & 10 [[[[\e_0, \e_1], \e_1], \e_1], \e_1] + 10 [\e_1, [[[\e_1, \e_0], \e_0], \e_0]] - 10 [[\e_1, \e_0], [[\e_1, \e_0], \e_0]]\\
        -&10 [[[[\e_1, \e_0], \e_0], \e_0], \e_0] - 10 [\e_0, [[[\e_0, \e_1], \e_1], \e_1]] + 10 [[\e_0, \e_1], [[\e_0, \e_1], \e_1]]
        .
    \end{align*}
\end{ex}

\begin{defn} \label{defn: phielementdefn}
By extending $\phi$ linearly, we may define canonical elements\footnote{In the notation $\tau_{2k+1}$ of \cite{RossiWillwacher:Etingof}, these elements are $\phi_{2k+1}=(2\ipi)^{2k+1}\tau_{2k+1}$.}
\[  \phi_{2k+1} =  \phi(\cochain^{\can}_{4k+1}) = \phi(\cochain^{\RW}_{4k+1}) \, \in \, \Lie(\e_0,\e_1)\otimes\RR\]
of length $2k+1$ for every $k\geq 1$. 
\end{defn}
We will later appeal to \cite[Theorem~1.4]{RossiWillwacher:Etingof}, which is about the elements $\phi(\cochain^{\RW}_{4k+1})$. Thanks to theorem~\ref{thm:RWeqCan}, these are in fact results about the canonical cocycles from $\omega^{4k+1}$.

\section{Motivic Lie algebra and \texorpdfstring{$\mathbb{P}^1 \setminus \{0,1,\infty\}$}{P\^{}1\textbackslash\{0,1,infinity\}}} \label{sec:motivics}

\subsection{Background on the projective line minus three points}
For further details and references of the material in this section, see \cite[\S2]{Brown:ICM2014}, \cite[\S4]{brownSVMZV}, and \cite{deligneP1}.

\subsubsection{de Rham}
Let $X= \mathbb{P}^1\setminus \{0,1,\infty\}$. Deligne's de Rham fundamental torsor
\[
    \opi \defas \pi_1^{\dR} ( X, \tone_0 ,  -\tone_1  ) 
\]
of paths on $X$ from the tangent vector $1$ at $0$ to the tangent vector $-1$ at $1$ is an affine scheme over $\QQ$. Its affine ring is the graded tensor coalgebra (with deconcatenation coproduct)
\begin{equation*}
    \mathcal{O}( \opi) 
    \cong  T (H^1_{\dR}( X) )  
    \cong T(\QQ \e_0 \oplus \QQ \e_1) 
    = \bigoplus_{w \in \{\e_0,\e_1\}^{\times} } w \, \QQ
\end{equation*}
which is generated by words $w$ in the letters $\e_0,\e_1$, including the empty word, and graded by length. We will refer to the length also as \emph{weight}, as it is the Hodge-theoretic weight divided by two. The generator $\e_0$ (resp.\ $\e_1$) corresponds to the de Rham cohomology class of the logarithmic form $\frac{\td x}{x}$ (resp.\ $\frac{\td x}{1-x}$) on $X$. 
Equipped with the weight grading and the commutative shuffle product, $ \mathcal{O}( \opi) $ is a graded and commutative $\QQ$-algebra. Its unit is the empty word. Picking out the coefficient of the empty word defines an augmentation $\mathcal{O}( \opi)\rightarrow \QQ $ and thus a point of $\opi$. It is called the canonical de Rham path and denoted
\[ \dchDR \in \opi(\QQ). \]
For any commutative ring $R$, any $R$-point $S\colon \mathcal{O}( \opi)\longrightarrow R $ of $\opi$ can be encoded in a formal power series in two non-commuting variables $\x_0,\x_1$ dual to $\e_0,\e_1$ of the form
\[  S=  \sum_{w\in \{\x_0,\x_1\}^{\times}} S_w\, w\]
where $S_w=S(w|_{\x\mapsto\e})$ is the value of $S$ on the word $w$ (viewed in the letters $\e_0,\e_1$). That $S$ is a homomorphism for the shuffle product implies constraints among the coefficients $S_w$. These can be written as $\Delta S=S\hat{\otimes} S$ in terms of the completed coproduct for which $\x_0$ and $\x_1$ are primitive. Therefore, the $R$-points of $\opi$ are the group-like elements $\opi(R)\cong\{S\colon \Delta S=S\hat{\otimes}S\}\subset R \llangle \x_0,\x_1\rrangle^{\times}$.

\subsubsection{Betti}
The Betti fundamental torsor of paths $\pi_1^{B} ( X, \tone_0 ,  -\tone_1  )$ is also an affine scheme over $\QQ$, and it is a bitorsor over the unipotent completions of the fundamental groups based at both of these tangential base points. There is a canonical map from the topological fundamental group into its set of rational points: 
\[  \pi_1 ( X(\CC), \tone_0 ,  -\tone_1  )\To \pi_1^{B} ( X, \tone_0 ,  -\tone_1  )(\QQ) \ . \]
The straight line (`droit chemin') from $0$ to $1$ defines an element in the left-hand side; we denote its image in the right-hand side by $\dchB$. 
The comparison isomorphism 
\[  \comp\colon \pi_1^{B} ( X, \tone_0 ,  -\tone_1  )(\CC) \overset{\sim}{\To} \opi(\CC) \  \]
is given explicitly by computing iterated integrals in the one forms $\frac{\td x}{x}, \frac{\td x}{1-x}$ along (completions) of topological paths. The image of $\dchB$ under this map is
\begin{equation}  \label{Zassocdef}
    \mathcal{Z} \defas \comp \big( \dchB \big) 
    =  \sum_{w\in \{\x_0,\x_1\}^{\times}} w \, \zeta(w)   \in \RR \llangle \x_0,\x_1 \rrangle
\end{equation}
where $\zeta(w)$ is the shuffle-regularised iterated integral from $0$ to $1$. It is a rational linear combination of multiple zeta values. In particular, for $w$ starting in $\x_1$ and ending in $\x_0$,
\begin{equation*}
    \zeta(\x_1\underbrace{\x_0\ldots\x_0}_{n_1-1}\x_1\underbrace{\x_0\ldots\x_0}_{n_2-1}\ldots\x_1\underbrace{\x_0\ldots\x_0}_{n_d-1})
    = \sum_{1\leq k_1<\cdots<k_d} \frac{1}{k_1^{n_1}\cdots k_d^{n_d}}
    .
\end{equation*}
The generating series $\mathcal{Z}$ is Drinfeld's associator. 

\subsubsection{Ihara action}\label{sec:Ihara-torsor}
The \emph{Ihara group law} on $\opi$ is defined by
\begin{eqnarray}\label{eq:Ihara-group}
\circ\colon  \opi \times \opi    &  \To &     \opi   \\ 
S \circ Z & = &  Z( \x_0, S  \x_1  S^{-1})\, . \, S \nonumber
\end{eqnarray}
and gives $(\opi,\circ)$ the structure of a group (the unit is $\dchDR$). Thus $\opi$ is also a torsor for this group, when we view $\circ$ as a left-action. In particular, for any two points $Z,Z'$ of $\opi$ there exists a \emph{unique} element $S$ in the group $(\opi,\circ)$ such that $S\circ Z=Z'$.

Infinitesimally, the Ihara group law induces a new Lie algebra structure on the graded Lie algebra $\Lie(\x_0,\x_1)$ of $\opi$. This \emph{Ihara bracket} is
\begin{equation}\label{eq:Ihara-bracket}
    \{ \sigma, \tau \} = [\tau,\sigma] - D_\sigma(\tau) + D_\tau(\sigma)
\end{equation}
where $D_\sigma$ denotes the derivation of $\Lie(\x_0,\x_1)$ such that $D_\sigma(\x_0)=0$ and $D_\sigma(\x_1)=[\x_1,\sigma]$.

\subsubsection{Motivic Lie algebra} \label{sect: MotivicLie}
The category $\MT(\ZZ)$ of mixed Tate motives over the integers \cite{LevineMotives, delignegoncharov} is a Tannakian category over $\QQ$ whose Tannaka group with respect to the de Rham fiber functor $\omega_{\dR}$ will be denoted by $G^{\dR} = \mathrm{Aut}^{\otimes}_{\MT(\ZZ)}(\omega_{\dR})$.
Let $U^{\dR}$ denote its unipotent radical. Since the fiber functor $\omega_{\dR}$ is graded, one deduces that 
\begin{equation} \label{GdRsplits} 
    G^{\dR} \cong U^{\dR} \rtimes \mathbb{G}_m\ .
\end{equation}
The graded Lie algebra of $U^{\dR}$ will be  denoted  by $\gm$ and will simply be   called  the `motivic Lie algebra' for short. 
The scheme $P^{B,\dR} = \mathrm{Isom}^{\otimes}_{\MT(\ZZ)}(\omega_{\dR},\omega_B)$ of isomorphisms of fiber functors from $\omega_{\dR}$ to $\omega_B$  is a torsor: 
\begin{equation} \label{Isomtorsor} G^{\dR} \times  P^{B,\dR}    \To   P^{B,\dR}\ . 
\end{equation}
The comparison isomorphism defines a canonical complex point 
\[ \comp_{B,\dR} \in  P^{B, \dR} (\CC)\ . \]

\begin{remark}  \label{remark: leftandright} In this paper motivic Galois groups will act on the \emph{right} of de Rham realisations of motives, so that it acts on the left on schemes in the category of motives, which is the more common convention. 
This issue is often overlooked in the literature but can lead to contradictory formulae if one is not careful. Thus \eqref{Isomtorsor} denotes the left-action obtained by precomposing an element in $P^{B,\dR}$ with the right-action by an element of $G^{\dR}$.
\end{remark}

Since the vector space $\mathcal{O}(\opi)$ is the de Rham realisation of an ind-object in the category $\MT(\ZZ)$, it follows from the Tannaka theorem that there is a (right-) action of $G^{\dR}$ on $\mathcal{O}( \opi)$, or equivalently, a (left-) action $G^{\dR} \times  \opi   \rightarrow \opi$. By acting on the canonical de Rham path $\dchDR$ we deduce a morphism of affine schemes over $\QQ$:
\begin{equation} \label{Gdrtoopi}
    g\mapsto g. \dchDR\colon G^{\dR}  \To  \opi \ .  
\end{equation}
With respect to the Ihara group law \eqref{eq:Ihara-group}, this map $G^{\dR} \rightarrow (\opi, \circ)$ is a morphism of affine group schemes. Restricting to $U^{\dR}$ and passing to graded Lie algebras, we obtain a map 
\begin{equation} \label{gmtoLie}
    \iota\colon \gm \To   ( \Lie(\e_0,\e_1) , \{  \ ,  \ \})
\end{equation}
of graded Lie algebras where $\{, \}$ is the Ihara Lie bracket \eqref{eq:Ihara-bracket}. Deligne predicted that $\iota$ is injective. This was established in \cite{BrMTZ}. Consequently, we can identify $\gm$ with its image under $\iota$ to save having to write $\iota \gm$ systematically when it is clear from the context. 
 
The structure of the category $\MT(\ZZ)$ implies, via Borel's calculation of the rational algebraic $K$-theory of $\ZZ$, that $\gm$ is a free graded Lie algebra with one generator in every odd degree $2k+1$, for $k\geq 1$. For simplicity of exposition, let us choose generators $\sigma_{2k+1}$ in degree $2k+1$ once and for all. Thus we have
\begin{equation*}
    \frac{\gm}{[\gm,\gm]} \cong H_1(\gm) \cong  \bigoplus_{k\geq 1} [\sigma_{2k+1}] \QQ  \ .
\end{equation*}
 The images $[\sigma_{2k+1}]$ of the generators in the abelianisation of $\gm$ are canonical,  the elements $\sigma_{2k+1}$ themselves are not. 
Normalized in the standard way, they satisfy
\begin{equation} \label{sigmas}
    \iota(\sigma_{2k+1}) =  \ad(\e_0)^{2k} (\e_1) + \hbox{ terms with two or more } \e_1\hbox{'s} \ .
\end{equation}

\subsubsection{Depth filtration} The  arguments which follow  can be carried out using abelianisations but a more surgical approach is  to use the depth filtration. It is denoted $ D^i\,  \Lie(\e_0,\e_1)$ and is the decreasing filtration associated to the depth-degree, for which $\e_0$ has depth $0$ and $\e_1$ has depth $1$. Thus  $ D^0=  \Lie(\e_0,\e_1)$, and $D^i$ is generated by Lie brackets of elements with $\geq i$ occurrences of $\e_1$. One can show that the depth filtration is preserved by the Ihara bracket and induces a filtration, also denoted by  $D^i$,  on $\gm$ via \eqref{gmtoLie}.
It satisfies $D^1 \gm= \gm$ as can be seen from \eqref{sigmas}. 
%
%
Consider the natural map 
\begin{equation} \label{Depth1iota} \gr^1_D \iota\colon  \gr^1_D  \gm \To \gr^1_D   \Lie(\e_0,\e_1) \cong \bigoplus_{n\geq 0} \ad(\e_0)^{n}\e_1 \QQ\ . \end{equation} 
Since $[\gm, \gm] \subset D^2 \gm$, it  it induces an injective map 
\begin{equation}\label{H1gmtodepth1} H_1(\gm)   \To      \bigoplus_{k\geq 1} \ad(\e_0)^{2k}\e_1 \QQ  \end{equation} 
which sends  ${[}\sigma_{2k+1}]$ to $\ad(\e_0)^{2k}\e_1$ for every $k\geq 1$. 

\subsection{Motivic associators}
The scheme of motivic associators is implicitly defined in \cite{BrMTZ, brownSVMZV} but not given a name. For the sake of clarity we give a different presentation of  these objects  here and spell out some of their properties. 

\begin{defn} Let $\mathbb{A} \leq \opi$ and $\mathbb{H} \leq \opi$ be the subschemes defined by the Zariski-closures of the following orbits under $P^{B,\dR}(\QQ)$ or $G^{\dR}(\QQ)$, respectively:  
\[ \mathbb{H} =  \overline{ P^{B, \dR}(\QQ)\, \dchB} \quad \hbox{ and }  \quad \mathbb{A} =  \overline{ G^{\dR}(\QQ)\, \dchDR}\ .\]
\end{defn}
Another, more  intuitive, way to define these schemes is to  use the notion of motivic  and de Rham periods. For any $w\in\mathcal{O}(\opi) $ a word in $\e_0,\e_1$, define  
\begin{equation*}
    \zeta^{\mm}(w) = \big[ \mathcal{O}(\opi),  \dchB, w \big]^{\mm}
    \qquad \hbox{ and } \qquad
    \zeta^{\dR}(w) = \big[ \mathcal{O}(\opi),  \dchDR, w \big]^{\dR} \ .
\end{equation*}
The former is called a  motivic multiple zeta value. It is defined to be the  element  in  $\mathcal{O}(P^{B,\dR})$  given by the map $\phi\mapsto \dchB(\phi(w))$ for any  isomorphism of fiber functors $\phi \in P^{B,\dR}$. The latter, called a de Rham multiple zeta value, is the element of $\mathcal{O}(G^{\dR})$ defined by $g\mapsto  \dchDR(gw)$. 
 The motivic (resp. de Rham) multiple zeta values satisfy algebraic relations called the motivic (resp. de Rham) relations and generate a graded ring we shall  denote by $\mathcal{P}^{\mm}$ (resp. $\mathcal{P}^{\dR}$).  
Consider the surjective $\QQ$-linear maps 
\begin{eqnarray} \label{mapopitomotivicz} \mathcal{O}(\opi)\cong T (\e_0 \QQ \oplus \e_1 \QQ)  &\To &  \mathcal{P}^{\bullet}     \\ 
w  & \mapsto & \zeta^{\bullet}(w)  \nonumber 
\end{eqnarray}
where $\bullet \in \{\mm, \dR\}$. Their kernels  are  the graded  ideals $\mathcal{J}^{\mm}$ (resp.   $\mathcal{J}^{\dR}$) of motivic (resp. de Rham) relations. 
It follows from this description that  $\mathbb{A}, \mathbb{H} \leq \opi$ satisfy 
\[     \mathbb{H}=\mathrm{Spec} \, (\mathcal{O}(\opi)/\mathcal{J}^{\mm})   \qquad   \hbox{ and } \qquad 
 \mathbb{A}=\mathrm{Spec} \, (\mathcal{O}(\opi)/\mathcal{J}^{\dR})  \ .  \]
 Thus $\mathbb{H}$ is precisely the  affine algebraic variety defined by all motivic relations between motivic multiple zeta values. 
 Equivalently, since   \eqref{mapopitomotivicz} is surjective, one can view $\mathcal{P}^{\mm}$  (resp.  $\mathcal{P}^{\dR}$)  as the affine ring of $\mathbb{H}$ (resp.  $\mathbb{A}$). 
One proves  that the ideal $\mathcal{J}^{\dR}$ is  generated by  $J^{\mm}$ together with  the single additional  relation $\zeta^{\dR}(2)=0$.

 \begin{remark} Evaluation on $\comp_{B,\dR} \in P^{B,\dR}(\CC)$ defines a canonical period homomorphism $\mathrm{per}\colon \mathcal{P}^{\mm}\rightarrow \RR$ which sends $\zeta^{\mm}(w)$ to $\zeta(w)$, the shuffle-regularised multiple zeta value. In this way the ideal of motivic relations define actual relations between multiple zeta values. It is conjectured that all relations between multiple zeta values are motivic.
 \end{remark}

The scheme $\mathbb{H}(\mathcal{P}^{\mm})$ contains a canonical point corresponding to the  map $\mathcal{O}(\opi)\rightarrow \mathcal{P}^{\mm}$ given by \eqref{mapopitomotivicz}. It is called the motivic Drinfeld associator $\mathcal{Z}^{\mm} \in \mathbb{H}(\mathcal{P}^{\mm})$ and satisfies
\begin{align*}
    \mathcal{Z}^{\mm} 
    &= \sum_{w\in \{\x_0,\x_1\}^{\times}} \zeta^{\mm}(w)\,w 
    =  1 + \zeta^{\mm}(2) [\x_1,\x_0] 
    +\zeta^{\mm}(3)\big([\x_0,[\x_0,\x_1]]+[\x_1,[\x_1,\x_0]]\big)
    + \ldots
\end{align*}
The image of $\mathcal{Z}^{\mm}$ under the period homomorphism is the Drinfeld associator $\mathcal{Z} = \mathrm{per} (\mathcal{Z}^{\mm})$. 
Similarly, the map   $\mathcal{O}(\opi)\rightarrow \mathcal{P}^{\dR}$ given by \eqref{mapopitomotivicz} defines an element 
\[ \mathcal{Z}^{\dR} =  \sum_{w\in \{\x_0,\x_1\}^{\times} } \zeta^{\dR}(w) w   \quad \in \quad \mathbb{A}(\mathcal{P}^{\dR}) \ .  \]
To lowest order in $\x_1$ one has  
\begin{equation} \label{ZdRexpansion} \mathcal{Z}^{\dR}= 1+ \sum_{k \geq 1} \zeta^{\dR}(2k+1) \, \ad(\x_0)^{2k}(\x_1) +  \hbox{ terms of degree } \geq 2 \hbox{ in } \x_1 
\end{equation} using the the shuffle relations for motivic (and hence de Rham) iterated integrals, and the fact that  all even de Rham zeta values $\zeta^{\dR}(2n)$ vanish \cite{BrMTZ}. A  choice of generators $\sigma_{2k+1}$ for $\gm$ may be obtained from $\mathcal{Z}^{\dR}$ by reading off the coefficients of $\zeta^{\dR}(2k+1)$ with respect to a suitable choice of basis of  $\mathcal{Z}^{\dR}$.

We have a commutative diagram (see remark \ref{remark: leftandright}):  
\begin{equation} \label{commdiag}
\begin{tikzcd}[column sep=large]
G^{\dR} \times P^{B,\dR}    \arrow[r] \arrow[d] &
P^{B,\dR} \arrow[d] \\
\mathbb{A}\times \mathbb{H}   \arrow[r] \arrow[d,hook] &
\mathbb{H} \arrow[d,hook] \\
\opi \times \opi \arrow[r, "\circ"] &
\opi
\end{tikzcd}
\end{equation}
The vertical maps on the bottom are closed immersions of subschemes. The map along the bottom is the Ihara action.  Note that $\mathbb{H} $ is not a torsor over $\mathbb{A}$, since on the level of affine rings one has 
\[\mathcal{O}(\mathbb{H}) = P^{B,\dR} \cong \mathcal{P}^{\dR} \otimes_{\QQ} \QQ[t^2] \cong \mathcal{O}(\mathbb{A}) \otimes_{\QQ} \QQ[t^2]\ , \] where $t$ is the coordinate on  the semi-simple quotient $\mathbb{G}_m$ of $G^{\dR}$.   The element $t^2$ corresponds to $24 \,\zeta^{\mm}(2)$ (or  equivalently the motivic version of $(2\pi i)^2$ by \cite[\S2.3]{BrMTZ}).

\begin{lem}\label{lemlogZdR} One has 
\[ \mathbb{A} = \overline{ U^{\dR}(\QQ) \dchDR}\ . \]
In particular the element $\mathcal{Z}^{\dR} \in \mathbb{A}(\mathcal{P}^{\dR}) \leq \opi(\mathcal{P}^{\dR})$ is in the image of the pro-unipotent radical $U^{\dR}$. We may therefore take its logarithm with respect to the Ihara group law $\circ$, which defines an element in the completed Lie algebra in $\x_0,\x_1$:
\[ \log^{\circ} \mathcal{Z}^{\dR}\quad  \in \quad  \widehat{\Lie(\x_0,\x_1)}\, \widehat{\otimes}_{\QQ} \, \mathcal{P}^{\dR}  \]
Its depth-graded component in weight $2k+1$ is in the image of  $\gr^1_D  \gm\otimes \gr^W_{2k+1} \mathcal{P}^{\dR}$ under the map  defined in equation \eqref{Depth1iota}, more precisely, for all $k\geq 1$:  
\[  \gr^W_{2k+1} \gr^1_D  \log^{\circ} \mathcal{Z}^{\dR}     =   \zeta^{\dR}(2k+1)\, \ad(\x_0)^{2k} \x_1  \ .   \]
\end{lem}

\begin{proof} The action of an element  $\lambda$ in the semi-simple quotient $\mathbb{G}_m$ on $\mathcal{O}(\opi)$ is  by multiplication by $\lambda^w$ in graded weight $w$. The element $\dchDR$ is defined by the augmentation map 
$\epsilon\colon \mathcal{O}(\opi)\rightarrow \gr^W_0\mathcal{O}(\opi)\cong  \QQ $ which is projection on to graded weight $0$. It follows that $\mathbb{G}_m$ acts trivially on $\dchDR$ and so its $U^{\dR}$-orbit and $G^{\dR}$-orbits coincide. Alternatively, one could also argue using the fact that $\zeta^{\dR}(2)=0$. 
Next, use the fact that for a pro-unipotent affine group scheme $U$ over a field of characterstic zero with pro-nilpotent Lie algebra $\mathfrak{u}$, the exponential map $\exp\colon \mathfrak{u} \overset{\sim}{\rightarrow} U$ is an isomorphism of affine schemes. Its inverse is the logarithm.   Now, since by \eqref{Gdrtoopi} the Ihara action and group multiplication  in $G^{\dR}$ coincide, the logarithm of $\mathcal{Z}^{\dR}$ exists and is the image of  an element in the Lie algebra of $U^{\dR}$ (which is the completion of the graded Lie algebra $\gm$), with coefficients in $\mathcal{P}^{\dR}.$ 
    For the last part, use the expansion \ref{ZdRexpansion}  and the fact that $\{D^1, D^1\} \subset D^2$ in the Lie algebra of $\opi$, which implies that 
    \[ \exp^{\circ} (\alpha) \equiv  1 + \alpha  \pmod{\hbox{terms of degree } \geq 2 \hbox{ in } \x_1}\   \]
    for any $\alpha  \in  \widehat{\Lie(\x_0,\x_1)}$, 
    where $\exp^{\circ}$, which is the inverse to $\log^{\circ}$,  is the exponential with respect to the Ihara group law $\circ$. 
\end{proof}

\subsection{Single-valued periods}  The following material is discussed in more detail in \cite{brownSVMZV}. 
Let $\tau\colon \mathbb{G}_m \rightarrow G^{\dR}$ denote the splitting \eqref{GdRsplits}.  
Consider the element 
\[ \comp^{\sigma}= \tau(-1) \circ \overline{\comp}   \quad \in \quad P^{B,\dR}(\CC)\ . \]
Since $\mathcal{Z}$ has real coefficients, its image in $\mathbb{H}(\CC)$ is 
\[ \mathcal{Z}^{\sigma} =  \mathcal{Z}(-\x_0, -\x_1)  \quad \in \quad \mathbb{H}(\CC)\ .  \]
Note that if $F_{\infty}\colon M_B \rightarrow M_B$ is the real Frobenius involution on the Betti realisation of any mixed Tate motive $M$, then we can also  write  $\comp^{\sigma} = \tau(-1)  \circ \comp \circ F_{\infty}. $

\begin{defn}
Since $P^{B,\dR}$ is a torsor over $G^{\dR}$ by \eqref{Isomtorsor}, there exists a unique point $\sv \in G^{\dR}(\CC)$ such that 
\[  \sv . \comp^{\sigma}    = \comp \]  
\end{defn}
One shows in fact that $\sv$ is invariant under complex conjugation and hence  $\sv \in G^{\dR}(\RR)$ is real-valued. The element $\sv$ defines a single-valued period homomorphism
\begin{eqnarray} \label{svperiodhom} \sv\colon  \mathcal{P}^{\dR}  & \To &  \RR  \\ 
\zeta^{\dR}(w) & \mapsto & \zeta^{\sv}(w) \nonumber 
\end{eqnarray} 
where $\zeta^{\sv}(w)$ is the so-called single-valued multiple zeta value. One shows furthermore that $\sv \in U^{\dR}(\RR)$ since the image\footnote{This is precisely the purpose of the twist by $\tau(-1)$ in the definition of $\comp^{\sigma}$. Note that the splitting  \eqref{GdRsplits} is a particular feature of mixed Tate motives and does not generalise. Therefore the element $\sv$ is not to be confused with its version without the twist by $\tau(-1)$, which exists in much more general situations, and which we usually denote by $\mathsf{s}$. It is not pro-unipotent.  } of $\sv$ in $\mathbb{G}_m(\RR)$ is $1$. We define
\[ \mathcal{Z}^{\sv} = \sv ( \mathcal{Z}^{\dR}) = \sum_w \zeta^{\sv}(w) w  \  \in \  \mathbb{A}(\RR)\leq \opi(\RR)\ . \]
It follows from the diagram \eqref{commdiag} (see also \cite[\S5]{brownSVMZV}) that
\begin{equation}  \label{ZsvcircZ}
    \mathcal{Z}^{\sv} \circ \mathcal{Z}^{\sigma} = \mathcal{Z}\ .
\end{equation}
Furthermore $\mathcal{Z}^{\sv}$ is the unique element in $_0\Pi_1(\RR)$ which satisfies this equation, see \S\ref{sec:Ihara-torsor}.

Finally, for all $n\geq 2$, one computes 
\begin{equation}  \label{zetasvn} 
    \zeta^{\sv}(n) =  
    \begin{cases}
        2 \, \zeta(n) & \text{if $n$ is odd,} \\
        0             & \text{if $n$ is even.} \\
\end{cases}
\end{equation}

\begin{cor} \label{cor: grlogZ} It follows from lemma \ref{lemlogZdR} that 
\[  \gr^W_{2k+1} \gr^1_D  \log^{\circ} \mathcal{Z}^{\sv}=  2 \, \zeta(2k+1)  \,\ad(\x_0)^{2k} \x_1 \  .  \]
\end{cor}

\section{The canonical cocycles generate the motivic Lie algebra}\label{sec:can-are-motivic}
Using the one-forms $\alpha_{2k}  =  t^{2k}(1-t)^{2k} \td t$ on the unit interval $[0,1]$ and the elements $\phi_{2k+1}\in\Lie(\e_0,\e_1)\otimes\RR$ from definition~\ref{defn: phielementdefn}, we can consider the path-ordered exponential
\[
    \psi=  \PO \exp^{\circ} \left( \sum_{k\geq 1} \int_0^1 \alpha_{2k} \phi_{2k+1} \right)  \in \opi(\RR)
\]
in $(\opi, \circ)$. Explicitly this is the element 
\begin{equation}  \label{psiexpansion} \psi = \sum_{r\geq 0} \sum_{k_1,\ldots, k_r \geq 0} \left(\int_0^1 \alpha_{2k_1}\ldots \alpha_{2k_r}\right) 
\phi_{2k_1+1} \circ \cdots \circ \phi_{2k_r+1} 
\end{equation}
where the composition is with respect to the Ihara group law $\circ$, and $  \int_0^1 \alpha_{2k_1}\ldots \alpha_{2k_r} \in \QQ$ is an iterated integral. 
Thus 
\[ \psi = 1 + \sum_{k\geq 1}  c_{2k}\,  \phi_{2k+1}  + \sum_{k,\ell \geq 1}  c_{2k,2\ell}\,   \phi_{2k+1} \circ \phi_{2\ell+1} + \ldots \]
where the $c_{2k}$ are given by the values of the beta function
\begin{equation} \label{c2formula} c_{2k} = \int_{0}^1 \alpha_{2k} = \beta(2k+1, 2k+1) \qquad \hbox{ and hence } \qquad  c_{2k}^{-1} = (4k+1) \binom{4k}{2k} \ . \end{equation} 
 We are now in a position to prove the main technical result.

\begin{thm} \label{thm: maintechnical} The logarithm $\log^{\circ} \psi$ exists and lies in the completion  $\widehat{\gmr}$ of the motivic Lie algebra $\gmr=\gm \otimes \RR$. Its weight-graded components    
satisfy 
\begin{equation}   \label{grWlogpsi} \gr^W_{2k+1} \log^{\circ} \psi \equiv   c_{2k}  \, \phi_{2k+1}       \pmod{ \gr^W_{2k+1}   [ \gmr, \gmr ]   } \end{equation}
in $\gr^W_{2k+1}  \gmr $. 
From this it follows that 
\begin{equation}  
\phi_{2k+1}  \ \in \ \gmr \ . 
\end{equation}
In other words, the elements $\phi_{2k+1}$ are in the image of the motivic Lie algebra. Furthermore, for every $k\geq 1$,  their classes in  $H_1(\gmr) = (\gm)^{\mathrm{ab}}\otimes_{\QQ} \RR$ satisfy
\[  [ \phi_{2k+1} ] =  -2 (4k+1) \binom{4k}{2k} \zeta(2k+1) \,  [\sigma_{2k+1}]    \ . \]
It follows that  the elements $\phi_{2k+1}$ are generators of $\gmr$. 
\end{thm}

\begin{proof} 
It is shown in \cite[Theorem~1.4]{RossiWillwacher:Etingof}, that\footnote{%
In \cite{RossiWillwacher:Etingof}, the associators are denoted $\Phi_{\text{KZ}}=\mathcal{Z}$ and $\Phi_{\overline{\text{KZ}}}=\mathcal{Z}^\sigma$ and written in letters $X$ and $Y$ which are related to ours by $\x_0=X/(2\ipi)$ and $\x_1=Y/(2\ipi)$.}
\[ \psi \circ \mathcal{Z} = \mathcal{Z}^{\sigma} \ . \]
It follows from \eqref{ZsvcircZ} and uniqueness that
\[ \psi = (\mathcal{Z}^{\sv})^{\circ-1} \ . \] 
As established in the previous section, $\mathcal{Z}^{\sv} \in \mathbb{A}(\RR)$ lies in the orbit of the unipotent radical of the motivic Galois group. Since $\psi = (\mathcal{Z}^{\sv})^{\circ-1}$ we can therefore take its logarithm $\log^\circ \psi$, which lies in the completion of the motivic Lie algebra $\widehat{\gm}\otimes \RR$. 

Next we may compute the logarithm of $\psi$ by applying the expansion due to Magnus \cite{Magnus:ExpSol,Chen1957} in terms of iterated integrals of commutators: 
\begin{equation*}
    \log^{\circ}  \left( \PO \exp^{\circ} \int_0^1 A(t) \td t  \right)
    =  \int_0^1 A(t) \td t +  \frac{1}{2} \int_{0\leq t_1 \leq t_2 \leq 1} [A(t_1), A(t_2)]\, \td t_1 \td t_2 +\ldots
\end{equation*}
where $A(t) =\sum_{k\geq 1} \alpha_{2k} \phi_{2k+1}$ is an element in the $\RR[[t]]$-points of the   Lie algebra of $\opi$. 
Since all higher terms  vanish in the  abelianisation, we deduce that 
\[  \log^{\circ}  \left( \PO \exp^{\circ} \int_0^1 A(t) \td t  \right) \equiv  \int_0^1 A(t) \td t   \pmod { \big[\widehat{\gmr}, \widehat{\gmr} \big]  }   \ .   \]
We conclude that
\[ \log^{\circ} \psi \equiv \sum_{k\geq 1} c_{2k} \phi_{2k+1}   \pmod { \big[ \widehat{\gmr}, \widehat{\gmr} \big]  }  \ . \]
Taking the weight-graded part in degree $2k+1$ leads to \eqref{grWlogpsi}.
Since $\log^\circ\psi= -\log^\circ \mathcal{Z}^{\sv}$ and $\gr^W_{2k+1} \log^{\circ} \mathcal{Z}^{\sv} \in \gmr$ is motivic, we deduce that $\phi_{2k+1} \in \gmr$.
The last displayed equation in the statement follows on  applying the map \eqref{Depth1iota} which   sends $\log^{\circ} \mathcal{Z}^{\sv}$ to $\zeta^{\sv}(2k+1) [\sigma_{2k+1}]$ by corollary  \ref{cor: grlogZ}. Conclude using \eqref{c2formula} and  \eqref{zetasvn}, which yields the factor of $2$. 
\end{proof}

\section{Geometry}\label{sec:geometry}
This section reviews constructions in the literature that trace a path from the cohomology of the general linear group $\GL_n(\ZZ)$, where the canonical forms naturally live, via the moduli space tropical abelian varieties and tropical curves, to the graph complex. Combining these with theorem~\ref{thm: maintechnical}, we arrive at our main results.

\subsection{Canonical forms and cohomology of \texorpdfstring{$\GL_n(\ZZ)$}{GLn(Z)}}
Let $g\geq 1$ and let $\SPD_g$ denote the space of positive definite symmetric matrices of rank $g$.
The group $\GL_g(\ZZ)$ acts on $\SPD_g$ on the right via $X \mapsto g^\Transpose X g$.  
Since $\SPD_g$ is connected and contractible, the quotient  $\SPD_g/\GL_g(\ZZ)$ is an orbifold $K(\pi,1)$ for the group $\GL_g(\ZZ)$.
The matrix entries constitute smooth functions $X_{ij}\colon \SPD_g\To\RR$ and give rise to differential forms $\td X_{ij}\in\Omega^1(\SPD_g)$.
From these we can construct, for every $n\geq 1$, the smooth differential form
\[
    \omega^{n}= \tr \big(X^{-1} \td X\big)^n
    \in \Omega^n(\SPD_g).
\]
It is closed, vanishes for $n \not\equiv 1 \pmod{4}$, and is invariant under the action of  $\GL_g(\RR)$; see \cite{BrInvariant}. In particular it can be interpreted as a differential form on the locally symmetric space $\SPD_g / \GL_g(\ZZ)$ and it defines a class $[\omega^n]\in H^n(\SPD_g/\GL_g(\ZZ),\RR)\cong H^n(\GL_g(\ZZ),\RR)$.

For fixed $n$, the classes defined by $\omega^n$ in different genera are pull-backs of each other, along the block-diagonal inclusion $\GL_g\hookrightarrow\GL_{g+1}$, $X\mapsto\big(\begin{smallmatrix} X&0\\0&1\\ \end{smallmatrix}\big)$. So if we consider the graded exterior algebra generated by formal symbols $\omega^n$ in every degree $n=4k+1\geq5$,
\begin{equation*}
    \Omega^{\bullet}_{\can}  =  \bigwedge\nolimits^\bullet  \bigoplus_{k\geq 1}  \QQ \, \omega^{4k+1}
    \ ,
\end{equation*}
then we obtain a canonical map
\[
    \Omega^n_{\can}\otimes_{\QQ} \RR  \To \varprojlim_{g} \,  H^n( \SPD_g/\GL_g(\ZZ);\RR)
    \ .
\]
A consequence of Borel's celebrated work \cite{Borel} is that this map is an isomorphism. The surjectivity follows from the work of Matsushima \cite{Matsushima} and Garland \cite{Garland};  Borel's contribution is the injectivity. From this result, one deduces that the rational algebraic $K$-theory of the integers $K_n(\ZZ)\otimes_{\ZZ} \QQ$ vanishes for $n$ not congruent to 1 modulo 4 and has rank one for $n=4k+1$, $k\geq 1$.  

More refined statements were recently established in \cite{BrBord}. The vector space $\Omega^{\bullet}_c$ of forms of \emph{compact type} is defined, again, as the vector space over $\QQ$ spanned by non-trivial exterior products of symbols $\omega^{4k+1}$ for $k \geq 1$ but this time 
 bigraded by genus and degree, where the genus of $ \omega^{4k_1+1} \wedge \cdots \wedge \omega^{4k_r +1}$, with $k_1 < \cdots < k_r$, is defined to be $2k_r+1$, the maximal genus of its factors. In \cite{BrBord} it is shown that there is a canonical injective map  
\begin{equation} \label{Omegactocompactcohom}
    \Omega^n_c  \otimes_{\QQ} \RR \To   \bigoplus_{g>0}  H_c^{n+1} ( \SPD_g/\GL_g(\ZZ);\RR)
\end{equation}
which maps classes of genus $g$ on the left into the summand of rank $g$ on the right. 
In particular, setting $r=1$, each $\omega^{4k+1}$ defines a non-trivial class
\begin{equation}   \label{omegacompacttypeclasses} 
    \big[\omega_c^{4k+1} \big]   \in  H_c^{4k+2}( \SPD_{2k+1} / \GL_{2k+1}(\ZZ); \RR)
    \ .
\end{equation}
The subscript $c$ denotes `compact support': these classes are not to be confused with the cohomology classes in $H^{4k+1}(\SPD_{g}/\GL_g(\ZZ);\RR)$ of the forms $\omega^{4k+1}$ themselves (see above). Instead, the classes $[\omega_c^{4k+1}]$ have canonical representatives as \emph{relative} cohomology classes that are constructed as follows: one first proves that the forms $\omega^{4k+1}$ extend to the Borel--Serre compactification of $\SPD_{g} / \GL_{g}(\ZZ)$. In the special case when $g=2k+1$, one furthermore finds that these forms vanish along its boundary. 

The following theorem establishes a non-commutative product on the compactly-supported cohomology of the locally symmetric spaces for $\GL_g(\ZZ)$, for $g\geq 0$.
\begin{thm}\label{thm:GL-Hopf}
    The $(k,g)$-bigraded vector space 
    \[  \mathcal{H}_c \defas \bigoplus_{g,k\geq 0}  H^{k+g}_c(\SPD_g/\GL_g(\ZZ); \RR) \]
    can be given the structure of a graded-cocommutative Hopf algebra.
\end{thm}
Such Hopf algebra structures have been constructed independently in \cite{AMP} and \cite{BCGP}. It is strongly suspected that these two Hopf algebras are one and the same. The coproduct arises from block-diagonal sum, and we expect that the multiplication agrees with the natural map on cohomology induced from the inclusion of the boundary of the (reductive) Borel--Serre compactification, as first suggested in \cite{LeeUnstable}.

It is shown in \cite[Theorem~1.2]{BCGP} that the image of \eqref{Omegactocompactcohom} is primitive for the coproduct on $\mathcal{H}_c$. Since the primitives in a cocommutative Hopf algebra form a Lie algebra under the commutator, any Hopf algebra structure as in theorem~\ref{thm:GL-Hopf} induces in particular a map from the free (bi-)graded Lie algebra on the elements \eqref{omegacompacttypeclasses}:
\[  \Lie(\omega_c^5, \omega_c^9,  \omega_c^{13}, \ldots ) \To  \Prim \, \mathcal{H}_c\ . \]
Our goal is to show to that this map is injective, and that its image is naturally isomorphic to the motivic Lie algebra $\gm$. Note that the image lies in the `diagonal' Hopf subalgebra
\begin{equation*}
    \mathcal{H}^{\Delta}_c \defas \bigoplus_{g\geq 0}  H^{2g}_c(\SPD_g/\GL_g(\ZZ); \RR) \ .
\end{equation*}

\subsection{Moduli spaces of tropical curves and abelian varieties}
Let $\mathcal{M}_g^{\trop}$ denote the moduli space of tropical curves \cite{BMV}, and $L\mathcal{M}_{g}^{\trop}$ its link. The latter is denoted $\Delta_g$ in \cite{CGP} and also arises as the quotient of the simplicial closure of Outer space by $\Out(F_n)$, see \cite[\S5.2]{ConantVogtmann:ThmKont}.
The link can be obtained out of the union of the projective simplices $\sigma_G$ from \eqref{eq:graph-simplex} over all stable weighted graphs $G$ of genus $g$,
\begin{equation*}
    L\mathcal{M}_{g}^{\trop} = \bigcup_{G} \, \sigma_G / {\sim}
\end{equation*}
where $\sim$ is the quotient by the action of the group of automorphisms of $G$ and $\sigma_{G/e}$ is identified with $\sigma_{G}\cap \{x_e=0\}$.

Let $\mathcal{A}_g^{\trop}$ denote the moduli space of tropical abelian varieties of rank $g$. It is defined to be $\SPD_{g}^{rt}/\GL_g(\ZZ)$, where $\SPD_g^{rt}$ is the rational closure $\SPD_g$ whose points are positive semi-definite matrices with rational kernel, equipped with the Satake topology. Let $L\mathcal{A}_{g}^{\trop}$ denote its link as defined in \cite[\S2]{TopWeightAg}. The open locus 
\[ \mathcal{A}_g^{\circ, \trop}  = \mathcal{A}_g^{\trop} \setminus \partial\mathcal{A}_g^{\trop} \]
is homeomorphic to the locally symmetric space $\SPD_g/\GL_g(\ZZ)$. 

The tropical Torelli map sends a stable metric graph $G$ of genus $g$ to the class of its graph Laplacian matrix, whose entries $x_e$ equal the length $\ell_e$ of each edge $e$: 
\begin{gather*}
    \lambda\colon  \mathcal{M}_g^{\trop} \To \mathcal{A}_g^{\trop} \\ 
    G \mapsto  \big[ \Lambda_G \oplus 0^{w(G)}\big]
\end{gather*}
where $w(G)$ denotes the sum of all vertex weights of $G$. 

In order to define a  tropical Torelli map on links, we must consider the reduced moduli space of tropical curves $\mathcal{M}^{\red}_g$ whose cells are indexed by 3-edge connected graphs (with no vertex weights). With this restriction, $\lambda$ induces a map  on links
$\lambda\colon  L\mathcal{M}_g^{\red} \rightarrow   L\mathcal{A}_g^{\trop}$. The  forms $\omega^{4k+1}$ do not define differential forms on the moduli space of tropical abelian varieties $\mathcal{A}^{\trop}_g$ since they have poles along the locus where $\det(X)$ vanishes.  Nevertheless, by blowing up boundary components at infinity one may construct bordifications  \cite{BrBord}
\[ \pi_{\mathcal{A}}\colon   L\mathcal{A}_g^{\trop,\mathbb{B}}\rightarrow L\mathcal{A}_g^{\trop}  
\qquad \ , \qquad
\pi_{\mathcal{M}}\colon   L\mathcal{M}_g^{\red,\mathbb{B}}\rightarrow L\mathcal{M}_g^{\red}
\]
and show that the pull-back $\pi_{\mathcal{A}}^* (\omega^{4k+1})$  of $\omega^{4k+1}$ extends to the boundary of $ L\mathcal{A}_g^{\trop,\mathbb{B}}$, along which it vanishes.
One can show \cite[Appendix]{BrBord} that $ L\mathcal{A}_g^{\trop,\mathbb{B}}$ is in fact homeomorphic to the Borel--Serre compactification \cite{BorelSerre}.  

The tropical Torelli map can be extended to the bordifications \cite{BrBord} and induces a map on their relative de Rham cohomology (in the sense of \emph{loc.\ cit.}): 
\begin{equation*}
    H^{n}_c ( L\mathcal{A}_g^{\circ, \trop}) 
    \cong H^{n}(  L\mathcal{A}_g^{\trop,\mathbb{B}}, \partial L\mathcal{A}_g^{\trop,\mathbb{B}})
    \overset{\lambda^*}{\To}  
    H^{n} ( L\mathcal{M}_g^{\red,\mathbb{B}}, \partial  L\mathcal{M}_g^{\red,\mathbb{B}})
\end{equation*}
where  $H^{n}_c (  L\mathcal{A}_g^{\circ, \trop})\cong H_c^{n+1}(\SPD_g/\GL_g(\ZZ))$. The map $\lambda^*$ sends the relative de Rham cohomology class $[\omega_c^{4k+1}]= ( \pi_{\mathcal{A}}^*(\omega^{4k+1}), 0)$ to the relative cohomology class $( \pi_{\mathcal{M}}^*(\omega^{4k+1}), 0)$. 
The latter pairs with relative homology 
\begin{equation*}
    H_n (  L\mathcal{M}_g^{\red,\mathbb{B}}, \partial  L\mathcal{M}_g^{\red,\mathbb{B}})
    \cong H_{n} (  L\mathcal{M}_g^{\red}, \partial  L\mathcal{M}_g^{\red}) 
    = H_{n+1-2g}((\hGC'_2)_g )
\end{equation*}
via the canonical integrals $G\mapsto I_G(\omega^{4k+1})$  of definition \ref{defn:Ican}. The first  isomorphism in the previous equation is excision, the second is a version of \cite[Proposition 6.8]{BrBord} where $\hGC_2'$ denotes a variant of the  graph complex for 3-edge connected graphs, and $(\hGC'_2)_g$ denotes its component in genus $g$.

We deduce a linear map $ H_c^{n+1}(\SPD_g/\GL_g(\ZZ);\RR) \rightarrow H^{n+1-2g}((\cGC'_2)_g )\otimes_{\QQ}\RR$. Its restriction to $n+1=g$ induces a morphism    $\Prim \, \mathcal{H}^{\Delta}_c  \otimes_{\QQ} \RR
\rightarrow H^0( \cGC'_2)\otimes_{\QQ} \RR $. It is shown in \cite[Proposition~1.8]{BCGP}  that it factors through a    morphism 
\begin{equation} \label{PrimHDeltatoGC}
    \lambda^*\colon 
    (\Prim \, \mathcal{H}^{\Delta}_c) \otimes_{\QQ} \RR
    \To H^0( \cGC_2)\otimes_{\QQ} \RR
\end{equation}
which is furthermore shown to be a morphism of graded Lie algebras, with respect to the Lie algebra structure on $\Prim\mathcal{H}_c^\Delta$ induced by the Hopf algebra structure on $\mathcal{H}_c$ defined in \cite{BCGP}. By the discussion above, it maps 
\begin{equation*}
    \lambda^*\colon \big[\omega_c^{4k+1}\big]  \mapsto   \big[\cochain^{\can}_{4k+1}\big]   \ . 
\end{equation*}

\subsection{Canonical classes in the cohomology of the general linear group}

Recall the map of Lie algebras
$
\phi\colon
H^0(\cGC_2)
\longrightarrow
(\Lie(\e_0,\e_1), \{\, , \,\})
$
from \eqref{eq:graphs-to-words}.
Composing with \eqref{PrimHDeltatoGC} gives a morphism of graded Lie algebras 
\[  \phi\lambda^* \colon  (\Prim \, \mathcal{H}^{\Delta}_c) \otimes_{\QQ} \RR
\To (\Lie(\e_0,\e_1), \{\, , \,\}) \otimes_{\QQ} \RR   \ .   \]
Recall from \eqref{gmtoLie} that the motivic Lie algebra injects $\iota\colon \gm \hookrightarrow  (\Lie(\e_0,\e_1), \{\, , \,\})$. 
\begin{thm} \label{thm: omegasinject}
    The classes $[\omega^{4k+1}_c]$ defined in \eqref{omegacompacttypeclasses} generate a free Lie algebra inside $\Prim(\mathcal{H}_c^{\Delta})$. They are mapped by $\phi\lambda^*$ into the image of the motivic Lie algebra $\gmr=\gm\otimes_{\QQ} \RR$. Furthermore, for any choice of normalised generators $\sigma_3,\sigma_5,\ldots$ of $\gm$, we have 
    \begin{equation*}
        \phi\lambda^* \left( [\omega_c^{4k+1}] \right)
        \equiv -2 (4k+1) \binom{4k}{2k} \zeta(2k+1) \, \sigma_{2k+1} 
        \pmod{[\gmr, \gmr]}
    \end{equation*}
    and hence the $\phi\lambda^* \big([\omega_c^{4k+1}]\big)$ are generators of $\gmr$. 
\end{thm}
\begin{proof}
This is a consequence of the discussion above, definition \ref{defn: phielementdefn} and theorem \ref{thm: maintechnical}. 
\end{proof}

\begin{remark}
    The proof shows more precisely that the image of $\phi\lambda([\omega_c^{4k+1}])$ lands in $\gm$ tensored with the $\QQ$-vector space generated by single-valued multiple zeta values of weight $2k+1$. 
\end{remark}

The theorem implies that there is an injective map of graded Lie algebras $\gm\hookrightarrow \mathcal{H}_c$.
\begin{cor} \label{cor: tensoralg}
Let $x$ be an element in bidegree $(1,0)$. There is an injective map of bigraded vector spaces from the tensor algebra
\[  T \Big( \bigoplus_{k\geq 1} \QQ \, \omega_c^{4k+1} [-1]\Big) \otimes_{\QQ} \QQ[x]/(x^2) \hooklongrightarrow \mathcal{H}_c \] 
where $\bigoplus_{k\geq 1} \QQ \, \omega_c^{4k+1} [-1]$ is the bigraded vector space in which $\omega_c^{4k+1}[-1]$ has cohomological degree $4k+2$ and genus $2k+1$ (thus it is of bidegree $(2k+1,2k+1)$). It sends $x$ to a generator of $H^1_c(\SPD_1/\GL_1(\ZZ))= H^1_c(\RR_{>0})$. 
\end{cor}
\begin{proof} This follows from  the Milnor--Moore  theorem for  connected, cocommutative graded Hopf algebras. See  the proof of 
   \cite[Corollary 7.1,  Remark 7.3]{BCGP}. 
\end{proof}
In fact, many other primitive classes are known (see \cite[\S9]{BCGP} and so one can deduce that a much larger algebra of canonical forms of compact type injects into $\mathcal{H}_c$). 

An interesting question is whether Theorem~\ref{intro: thm1} can be extended to show that the Lie algebra generated by \emph{all} canonical forms of compact type \eqref{Omegactocompactcohom} is also free. 

\subsection{Applications to the cohomology of \texorpdfstring{$\mathcal{A}_g$}{Ag}}
Let $\mathcal{A}_g$ be the moduli space of abelian varieties of dimension $g$. 
It was shown in \cite{BCGP} that the compactly-supported weight zero cohomology $ \bigoplus_g W_0  H_c^{\bullet+g}(\mathcal{A}_g)$ carries the structure of a bigraded graded-cocommutative Hopf algebra,  related to that on $\mathcal{H}_c$. In \cite[Theorem~1.2]{BCGP} it is shown that the canonical forms of compact type inject into the primitives of this Hopf algebra, and that the Lie algebra generated by the classes $[\omega_c^{4k+1}]$, \emph{for small $k$}, generate a free Lie algebra. This was achieved by showing that \eqref{PrimHDeltatoGC} factors through a map
\begin{equation} \label{PrimHAtoGC}
    \Prim \Big( \bigoplus_g  W_0 H^{2g}_c(\mathcal{A}_g)
    \Big) 
    \To H^0( \cGC_2)
    \ ,
\end{equation}
combined with Willwacher's theorem \cite{WillwacherGRT} that $H^0( \cGC_2)\cong\grt_1$ is the Grothendieck-Teichm\"uller Lie algebra, and exploiting computer calculations which verify that $\gm$ is isomorphic to $\grt_1$ in low weights.

We may now prove directly, and avoiding all mention of $\grt$, the following theorem, answering one half of \cite[Question~1.17]{BCGP}:
\begin{thm} \label{thm: Agfree} There is an embedding of the free graded Lie algebra
\begin{equation*}
    \Lie(\omega^{5}_c,\omega^9_c,\omega^{13}_c,\ldots) 
    \hooklongrightarrow
    \Prim \Big( \bigoplus_g  W_0 H^{2g}_c(\mathcal{A}_g; \RR) \Big) 
\end{equation*}
such that its  composition with the  map  
\begin{equation*}
    \Prim  \Big(  \bigoplus_g  W_0 H^{2g}_c(\mathcal{A}_g)\Big) 
    \To
    (\Lie(\e_0,\e_1), \{ \ , \} )
\end{equation*}
which follows from  composing \eqref{PrimHAtoGC} with  the map $\phi$ of \S\ref{sect: phi}, 
induces an isomorphism of graded Lie algebras: 
\[   \Lie(\omega^{5}_c,\omega^9_c,\omega^{13}_c,\ldots) \otimes_{\QQ} \RR \cong \gmr  \ . \]
\end{thm}

Note that this theorem is not a formal consequence of theorem~\ref{thm: omegasinject} alone---the fact quoted above that the classes $[\omega^{4k+1}_c]$ lift to the cohomology $\bigoplus_g  W_0 H^{2g}_c(\mathcal{A}_g)$ required understanding their precise behaviour with respect to the inflation sequence of \cite{TopWeightAg}.

\printbibliography

\appendix

\section{Explicit canonical cocycles}\label{sec:explicit-cocycles}

We computed explicit expressions for the canonical cocycles $\cochain^\can_{4k+1}\in\cGC_2\otimes\RR$ for $k\in\{1,2,3\}$. These are provided in table~\ref{tab:can357} and the ancillary files \texttt{cocycleH.txt} where $\texttt{H}=h_G=2k+1$ denotes the first Betti number of the corresponding graphs. These cocycles \eqref{eq:graph-cocycle-def} are linear combinations of oriented graphs such that $\langle \cochain^\can_{4k+1},G\rangle = I^\can_G$ under the identification $\cGC_2\cong\Hom(\hGC_2,\QQ)$, for all oriented graphs $G$ with $h_G=2k+1$. So
\begin{equation*}
    \cochain^\can_{4k+1}=\sum_{G\ \text{with}\ h_G=2k+1}
    \frac{I^\can_{[G,o]}}{|\Aut(G)|} [G,o]
\end{equation*}
is a sum over all isomorphism classes of graphs with $4k+2$ edges and $2k+2$ vertices, equipped with an arbitrary orientation (i.e.\ the sum is over a basis of $\gr_{2k+1}\hGC_2$). Any such graph which does not appear in the files (or table~\ref{tab:can357}) has $I^\RW_G=I^\can_G=0$.
In the files we encode oriented graphs as ordered edge lists $[e_1,\ldots,e_{4k+2}]$, with the implied orientation $o=e_1\wedge\ldots\wedge e_{4k+2}$. In table~\ref{tab:can357} we denote an edge $e=\{i,j\}$ simply as $ij$.

These cocycles were computed by calculating the integrals $I^\RW_G$ for each graph with the given number of vertices and edges. For this we adapted the {\kontsevint} code from \cite{BanksPanzerPym:MZVinDQ}, which is available at \url{https://bitbucket.org/PanzerErik/kontsevint/}. Since in weights $3,5$, and $7$ the space of single-valued multiple zeta values has dimension one, we find in each case that $I^\RW_G=q_G\cdot \zeta(2k+1)$ for some rational number $q_G\in\QQ$.

At $k=1$ there is only a single graph, the complete graph $K_4$ (wheel with 3 spokes). The cocycle $\cochain^\can_9$ at $k=2$ is also well-known and consists of two graphs; their canonical integrals were computed in \cite[\S10.2]{BrInvariant}. Our result for the next canonical cocycle $\cochain^\can_{13}$ is new. It is supported on 56 non-isomorphic graphs. It differs from the (non-canonical) cocycle $\boldsymbol\gamma_7$ provided in \cite{BuringKiselevRutten:Heptagon}.

\begin{table}
    {\scriptsize\begin{tabular}{lrr}
\toprule
\small $[G,o]$ & \small $|\Aut(G)|$ & \small $I^\can_{[G,o]}$ \\
\midrule
$12 \wedge 13 \wedge 14 \wedge 23 \wedge 24 \wedge 34$ & 24\ \ & $60 \cdot\zeta(3)$ \\
\midrule
$\llap{$Z_5=\,$} 13 \wedge 14 \wedge 15 \wedge 16 \wedge 24 \wedge 25 \wedge 26 \wedge 35 \wedge 36 \wedge 46$ & 2\ \ & $-630 \cdot\zeta(5)$ \\
$\llap{$W_5=\,$} 13 \wedge 14 \wedge 16 \wedge 24 \wedge 25 \wedge 26 \wedge 35 \wedge 36 \wedge 46 \wedge 56$ & 10\ \ & $1260 \cdot\zeta(5)$ \\
\midrule
$13 \wedge 15 \wedge 16 \wedge 18 \wedge 24 \wedge 26 \wedge 27 \wedge 28 \wedge 35 \wedge 37 \wedge 46 \wedge 47 \wedge 58 \wedge 68$ & 2\ \ & $9009 \cdot\zeta(7)$ \\
$13 \wedge 15 \wedge 16 \wedge 24 \wedge 26 \wedge 27 \wedge 28 \wedge 35 \wedge 37 \wedge 38 \wedge 46 \wedge 47 \wedge 58 \wedge 68$ & 1\ \ & $-3003 \cdot\zeta(7)$ \\
$14 \wedge 15 \wedge 16 \wedge 18 \wedge 25 \wedge 26 \wedge 27 \wedge 28 \wedge 36 \wedge 37 \wedge 38 \wedge 47 \wedge 48 \wedge 58$ & 1\ \ & $9009/2 \cdot\zeta(7)$ \\
$14 \wedge 15 \wedge 16 \wedge 18 \wedge 25 \wedge 26 \wedge 27 \wedge 36 \wedge 37 \wedge 38 \wedge 47 \wedge 48 \wedge 58 \wedge 68$ & 1\ \ & $-3003/2 \cdot\zeta(7)$ \\
$14 \wedge 15 \wedge 16 \wedge 18 \wedge 25 \wedge 26 \wedge 27 \wedge 36 \wedge 37 \wedge 38 \wedge 47 \wedge 48 \wedge 58 \wedge 78$ & 2\ \ & $-9009 \cdot\zeta(7)$ \\
$14 \wedge 15 \wedge 16 \wedge 18 \wedge 25 \wedge 27 \wedge 28 \wedge 36 \wedge 37 \wedge 38 \wedge 47 \wedge 48 \wedge 56 \wedge 78$ & 2\ \ & $3003 \cdot\zeta(7)$ \\
$14 \wedge 15 \wedge 16 \wedge 25 \wedge 26 \wedge 27 \wedge 28 \wedge 36 \wedge 37 \wedge 38 \wedge 47 \wedge 48 \wedge 58 \wedge 68$ & 2\ \ & $-6006 \cdot\zeta(7)$ \\
$14 \wedge 15 \wedge 17 \wedge 18 \wedge 24 \wedge 26 \wedge 27 \wedge 28 \wedge 35 \wedge 36 \wedge 37 \wedge 46 \wedge 48 \wedge 58$ & 2\ \ & $3003/2 \cdot\zeta(7)$ \\
$14 \wedge 15 \wedge 17 \wedge 18 \wedge 24 \wedge 26 \wedge 27 \wedge 28 \wedge 35 \wedge 36 \wedge 37 \wedge 47 \wedge 58 \wedge 68$ & 2\ \ & $-3003 \cdot\zeta(7)$ \\
$14 \wedge 15 \wedge 17 \wedge 18 \wedge 24 \wedge 26 \wedge 27 \wedge 28 \wedge 35 \wedge 36 \wedge 38 \wedge 46 \wedge 48 \wedge 57$ & 1\ \ & $-3003 \cdot\zeta(7)$ \\
$14 \wedge 15 \wedge 17 \wedge 18 \wedge 24 \wedge 26 \wedge 27 \wedge 28 \wedge 35 \wedge 36 \wedge 38 \wedge 46 \wedge 57 \wedge 58$ & 1\ \ & $-3003 \cdot\zeta(7)$ \\
$14 \wedge 15 \wedge 17 \wedge 18 \wedge 24 \wedge 26 \wedge 27 \wedge 35 \wedge 36 \wedge 37 \wedge 38 \wedge 46 \wedge 48 \wedge 58$ & 1\ \ & $3003 \cdot\zeta(7)$ \\
$14 \wedge 15 \wedge 17 \wedge 18 \wedge 24 \wedge 26 \wedge 27 \wedge 35 \wedge 36 \wedge 37 \wedge 38 \wedge 47 \wedge 58 \wedge 68$ & 1\ \ & $-3003 \cdot\zeta(7)$ \\
$14 \wedge 15 \wedge 17 \wedge 18 \wedge 24 \wedge 26 \wedge 27 \wedge 35 \wedge 36 \wedge 37 \wedge 46 \wedge 48 \wedge 58 \wedge 68$ & 2\ \ & $6006 \cdot\zeta(7)$ \\
$14 \wedge 15 \wedge 17 \wedge 18 \wedge 24 \wedge 26 \wedge 27 \wedge 35 \wedge 36 \wedge 38 \wedge 46 \wedge 57 \wedge 58 \wedge 68$ & 1\ \ & $3003 \cdot\zeta(7)$ \\
$14 \wedge 15 \wedge 17 \wedge 18 \wedge 25 \wedge 26 \wedge 27 \wedge 28 \wedge 36 \wedge 37 \wedge 38 \wedge 46 \wedge 48 \wedge 57$ & 1\ \ & $-3003/2 \cdot\zeta(7)$ \\
$14 \wedge 15 \wedge 17 \wedge 18 \wedge 25 \wedge 26 \wedge 27 \wedge 36 \wedge 37 \wedge 38 \wedge 46 \wedge 47 \wedge 58 \wedge 68$ & 2\ \ & $-3003/2 \cdot\zeta(7)$ \\
$14 \wedge 15 \wedge 17 \wedge 18 \wedge 25 \wedge 26 \wedge 27 \wedge 36 \wedge 37 \wedge 38 \wedge 47 \wedge 48 \wedge 58 \wedge 68$ & 1\ \ & $3003 \cdot\zeta(7)$ \\
$14 \wedge 15 \wedge 17 \wedge 18 \wedge 25 \wedge 26 \wedge 28 \wedge 36 \wedge 37 \wedge 38 \wedge 47 \wedge 48 \wedge 57 \wedge 68$ & 1\ \ & $-6006 \cdot\zeta(7)$ \\
$14 \wedge 15 \wedge 17 \wedge 18 \wedge 25 \wedge 26 \wedge 28 \wedge 36 \wedge 37 \wedge 38 \wedge 47 \wedge 48 \wedge 58 \wedge 67$ & 1\ \ & $3003 \cdot\zeta(7)$ \\
$14 \wedge 15 \wedge 17 \wedge 18 \wedge 25 \wedge 26 \wedge 28 \wedge 36 \wedge 37 \wedge 38 \wedge 47 \wedge 48 \wedge 58 \wedge 68$ & 1\ \ & $12012 \cdot\zeta(7)$ \\
$14 \wedge 15 \wedge 17 \wedge 24 \wedge 26 \wedge 27 \wedge 28 \wedge 35 \wedge 36 \wedge 37 \wedge 38 \wedge 46 \wedge 48 \wedge 58$ & 1\ \ & $-3003 \cdot\zeta(7)$ \\
$14 \wedge 15 \wedge 17 \wedge 24 \wedge 26 \wedge 27 \wedge 28 \wedge 35 \wedge 36 \wedge 37 \wedge 46 \wedge 48 \wedge 58 \wedge 78$ & 1\ \ & $-9009/2 \cdot\zeta(7)$ \\
$14 \wedge 15 \wedge 17 \wedge 24 \wedge 26 \wedge 27 \wedge 28 \wedge 35 \wedge 36 \wedge 38 \wedge 46 \wedge 48 \wedge 57 \wedge 58$ & 2\ \ & $3003 \cdot\zeta(7)$ \\
$14 \wedge 15 \wedge 17 \wedge 24 \wedge 26 \wedge 27 \wedge 35 \wedge 36 \wedge 37 \wedge 38 \wedge 46 \wedge 48 \wedge 58 \wedge 78$ & 1\ \ & $-3003/2 \cdot\zeta(7)$ \\
$14 \wedge 15 \wedge 17 \wedge 25 \wedge 26 \wedge 27 \wedge 28 \wedge 35 \wedge 36 \wedge 37 \wedge 38 \wedge 46 \wedge 48 \wedge 58$ & 4\ \ & $-6006 \cdot\zeta(7)$ \\
$14 \wedge 15 \wedge 17 \wedge 25 \wedge 26 \wedge 27 \wedge 28 \wedge 36 \wedge 37 \wedge 38 \wedge 47 \wedge 48 \wedge 58 \wedge 68$ & 1\ \ & $-3003 \cdot\zeta(7)$ \\
$14 \wedge 15 \wedge 17 \wedge 25 \wedge 26 \wedge 27 \wedge 36 \wedge 37 \wedge 38 \wedge 47 \wedge 48 \wedge 58 \wedge 68 \wedge 78$ & 2\ \ & $-9009 \cdot\zeta(7)$ \\
$14 \wedge 15 \wedge 17 \wedge 25 \wedge 26 \wedge 28 \wedge 36 \wedge 37 \wedge 38 \wedge 47 \wedge 48 \wedge 57 \wedge 58 \wedge 68$ & 1\ \ & $6006 \cdot\zeta(7)$ \\
$14 \wedge 15 \wedge 17 \wedge 25 \wedge 26 \wedge 28 \wedge 36 \wedge 37 \wedge 38 \wedge 47 \wedge 48 \wedge 57 \wedge 68 \wedge 78$ & 2\ \ & $-18018 \cdot\zeta(7)$ \\
$14 \wedge 15 \wedge 17 \wedge 25 \wedge 26 \wedge 28 \wedge 36 \wedge 37 \wedge 38 \wedge 47 \wedge 48 \wedge 58 \wedge 67 \wedge 78$ & 2\ \ & $-9009 \cdot\zeta(7)$ \\
$14 \wedge 15 \wedge 18 \wedge 25 \wedge 26 \wedge 27 \wedge 28 \wedge 36 \wedge 37 \wedge 38 \wedge 47 \wedge 48 \wedge 57 \wedge 68$ & 1\ \ & $-3003 \cdot\zeta(7)$ \\
$14 \wedge 15 \wedge 18 \wedge 25 \wedge 26 \wedge 27 \wedge 28 \wedge 36 \wedge 37 \wedge 38 \wedge 47 \wedge 48 \wedge 58 \wedge 68$ & 1\ \ & $-12012 \cdot\zeta(7)$ \\
$14 \wedge 15 \wedge 18 \wedge 25 \wedge 26 \wedge 28 \wedge 36 \wedge 37 \wedge 38 \wedge 47 \wedge 48 \wedge 58 \wedge 68 \wedge 78$ & 14\ \ & $-24024 \cdot\zeta(7)$ \\
$14 \wedge 16 \wedge 17 \wedge 18 \wedge 25 \wedge 26 \wedge 27 \wedge 28 \wedge 35 \wedge 37 \wedge 38 \wedge 46 \wedge 47 \wedge 58$ & 2\ \ & $-3003 \cdot\zeta(7)$ \\
$14 \wedge 16 \wedge 17 \wedge 18 \wedge 25 \wedge 26 \wedge 27 \wedge 28 \wedge 35 \wedge 37 \wedge 38 \wedge 46 \wedge 48 \wedge 58$ & 1\ \ & $-3003 \cdot\zeta(7)$ \\
$14 \wedge 16 \wedge 17 \wedge 18 \wedge 25 \wedge 26 \wedge 28 \wedge 35 \wedge 37 \wedge 38 \wedge 46 \wedge 47 \wedge 58 \wedge 68$ & 1\ \ & $-12012 \cdot\zeta(7)$ \\
$14 \wedge 16 \wedge 17 \wedge 18 \wedge 25 \wedge 26 \wedge 28 \wedge 35 \wedge 37 \wedge 38 \wedge 47 \wedge 48 \wedge 56 \wedge 58$ & 1\ \ & $-3003 \cdot\zeta(7)$ \\
$14 \wedge 16 \wedge 17 \wedge 18 \wedge 25 \wedge 26 \wedge 28 \wedge 35 \wedge 37 \wedge 38 \wedge 47 \wedge 48 \wedge 56 \wedge 68$ & 2\ \ & $9009 \cdot\zeta(7)$ \\
$14 \wedge 16 \wedge 17 \wedge 18 \wedge 25 \wedge 26 \wedge 28 \wedge 36 \wedge 37 \wedge 38 \wedge 47 \wedge 48 \wedge 57 \wedge 58$ & 1\ \ & $9009/2 \cdot\zeta(7)$ \\
$14 \wedge 16 \wedge 17 \wedge 25 \wedge 26 \wedge 27 \wedge 28 \wedge 35 \wedge 36 \wedge 37 \wedge 38 \wedge 47 \wedge 48 \wedge 58$ & 2\ \ & $3003 \cdot\zeta(7)$ \\
$14 \wedge 16 \wedge 17 \wedge 25 \wedge 26 \wedge 27 \wedge 28 \wedge 35 \wedge 37 \wedge 38 \wedge 46 \wedge 48 \wedge 58 \wedge 68$ & 1\ \ & $9009 \cdot\zeta(7)$ \\
$14 \wedge 16 \wedge 17 \wedge 25 \wedge 26 \wedge 27 \wedge 35 \wedge 36 \wedge 38 \wedge 47 \wedge 48 \wedge 57 \wedge 58 \wedge 68$ & 1\ \ & $-3003/2 \cdot\zeta(7)$ \\
$14 \wedge 16 \wedge 17 \wedge 25 \wedge 26 \wedge 27 \wedge 35 \wedge 37 \wedge 38 \wedge 46 \wedge 48 \wedge 57 \wedge 58 \wedge 68$ & 1\ \ & $-3003 \cdot\zeta(7)$ \\
$14 \wedge 16 \wedge 17 \wedge 25 \wedge 26 \wedge 27 \wedge 35 \wedge 37 \wedge 38 \wedge 47 \wedge 48 \wedge 56 \wedge 58 \wedge 68$ & 2\ \ & $-3003 \cdot\zeta(7)$ \\
$14 \wedge 16 \wedge 17 \wedge 25 \wedge 26 \wedge 28 \wedge 35 \wedge 37 \wedge 38 \wedge 47 \wedge 48 \wedge 56 \wedge 58 \wedge 78$ & 2\ \ & $9009 \cdot\zeta(7)$ \\
$14 \wedge 16 \wedge 17 \wedge 25 \wedge 26 \wedge 28 \wedge 36 \wedge 37 \wedge 38 \wedge 47 \wedge 48 \wedge 57 \wedge 58 \wedge 68$ & 1\ \ & $9009/2 \cdot\zeta(7)$ \\
$14 \wedge 16 \wedge 17 \wedge 25 \wedge 26 \wedge 28 \wedge 36 \wedge 37 \wedge 38 \wedge 47 \wedge 48 \wedge 57 \wedge 58 \wedge 78$ & 2\ \ & $9009/2 \cdot\zeta(7)$ \\
$14 \wedge 16 \wedge 18 \wedge 25 \wedge 26 \wedge 27 \wedge 28 \wedge 35 \wedge 36 \wedge 37 \wedge 38 \wedge 47 \wedge 48 \wedge 58$ & 4\ \ & $6006 \cdot\zeta(7)$ \\
$14 \wedge 16 \wedge 18 \wedge 25 \wedge 26 \wedge 27 \wedge 28 \wedge 35 \wedge 36 \wedge 37 \wedge 47 \wedge 48 \wedge 58 \wedge 68$ & 1\ \ & $9009 \cdot\zeta(7)$ \\
$14 \wedge 16 \wedge 18 \wedge 25 \wedge 26 \wedge 27 \wedge 28 \wedge 35 \wedge 37 \wedge 38 \wedge 47 \wedge 48 \wedge 56 \wedge 58$ & 1\ \ & $6006 \cdot\zeta(7)$ \\
$14 \wedge 16 \wedge 18 \wedge 25 \wedge 26 \wedge 27 \wedge 28 \wedge 35 \wedge 37 \wedge 38 \wedge 47 \wedge 48 \wedge 56 \wedge 68$ & 1\ \ & $-15015/2 \cdot\zeta(7)$ \\
$14 \wedge 16 \wedge 18 \wedge 25 \wedge 26 \wedge 27 \wedge 28 \wedge 35 \wedge 37 \wedge 38 \wedge 47 \wedge 48 \wedge 56 \wedge 78$ & 1\ \ & $-3003/2 \cdot\zeta(7)$ \\
$14 \wedge 16 \wedge 18 \wedge 25 \wedge 26 \wedge 27 \wedge 35 \wedge 37 \wedge 38 \wedge 47 \wedge 48 \wedge 56 \wedge 58 \wedge 68$ & 1\ \ & $-3003/2 \cdot\zeta(7)$ \\
$14 \wedge 16 \wedge 18 \wedge 25 \wedge 26 \wedge 27 \wedge 35 \wedge 37 \wedge 38 \wedge 47 \wedge 48 \wedge 56 \wedge 58 \wedge 78$ & 1\ \ & $9009/2 \cdot\zeta(7)$ \\
$14 \wedge 16 \wedge 18 \wedge 25 \wedge 26 \wedge 27 \wedge 35 \wedge 37 \wedge 38 \wedge 47 \wedge 48 \wedge 56 \wedge 68 \wedge 78$ & 1\ \ & $-6006 \cdot\zeta(7)$ \\
\bottomrule
\end{tabular}
}%
    \caption{Canonical graph cocycles $\cochain^\can_5$, $\cochain^\can_9$, and $\cochain^\can_{13}$.}%
    \label{tab:can357}%
\end{table}

The above canonical cocycles map through definition~\ref{defn: phielementdefn} to elements $\phi_{2k+1}=\phi(\cochain^\can_{4k+1})\in\gm\otimes \RR$. Since in these low degrees $\gm$ has rank one (there are no commutators yet), by theorem~\ref{thm: maintechnical} we know that these cocycles can be written as
\begin{equation*}
    \phi_3 = -60\zeta(3)\cdot\sigma_3,\qquad
    \phi_5 = -1260\zeta(5)\cdot\sigma_5,\qquad
    \phi_7 = -24024\zeta(7)\cdot\sigma_7
\end{equation*}
in terms of rational generators
$\sigma_3,\sigma_5,\sigma_7\in\gm$. We calculated these by explicitly evaluating the map $\phi$ on the canonical graph cocycles. This consistency check confirmed the well-known expressions for these rational generators:
\begin{align*}
	\sigma_3 &= [\e_0, [\e_0, \e_1]] - [[\e_0, \e_1], \e_1]\\ 
	\sigma_5 &= [\e_0, [\e_0, [\e_0, [\e_0, \e_1]]]] - 2 \cdot[\e_0, [\e_0, [[\e_0, \e_1], \e_1]]] + \tfrac{3}{2} \cdot [[\e_0, [\e_0, \e_1]], [\e_0, \e_1]]\\
             &+ 2 \cdot [\e_0, [[[\e_0, \e_1], \e_1], \e_1]] + \tfrac{1}{2} \cdot [[\e_0, \e_1], [[\e_0, \e_1], \e_1]] - [[[[\e_0, \e_1], \e_1], \e_1], \e_1] \\
	\sigma_7 &= [\e_0, [\e_0, [\e_0, [\e_0, [\e_0, [\e_0, \e_1]]]]]] - 3 \cdot [\e_0, [\e_0, [\e_0, [\e_0, [[\e_0, \e_1], \e_1]]]]] \\
    &+ 5 \cdot [\e_0, [\e_0, [[\e_0, [\e_0, \e_1]], [\e_0, \e_1]]]] + 5 \cdot [\e_0, [\e_0, [\e_0, [[[\e_0, \e_1], \e_1], \e_1]]]] \\
    &- 2 \cdot [[\e_0, [\e_0, [\e_0, \e_1]]], [\e_0, [\e_0, \e_1]]] + \tfrac{19}{16} \cdot [\e_0, [\e_0, [[\e_0, \e_1], [[\e_0, \e_1], \e_1]]]] \\
    &- \tfrac{173}{16} \cdot [\e_0, [[\e_0, [[\e_0, \e_1], \e_1]], [\e_0, \e_1]]] - 5 \cdot [\e_0, [\e_0, [[[[\e_0, \e_1], \e_1], \e_1], \e_1]]] \\
    &- 2 \cdot [[\e_0, [\e_0, \e_1]], [\e_0, [[\e_0, \e_1], \e_1]]] + \tfrac{17}{16} \cdot [[[\e_0, [\e_0, \e_1]], [\e_0, \e_1]], [\e_0, \e_1]] \\
    &- \tfrac{99}{16} \cdot [\e_0, [[\e_0, \e_1], [[[\e_0, \e_1], \e_1], \e_1]]] + \tfrac{61}{16} \cdot [[\e_0, [[\e_0, \e_1], \e_1]], [[\e_0, \e_1], \e_1]] \\
    &+ \tfrac{109}{16} \cdot [[\e_0, [[[\e_0, \e_1], \e_1], \e_1]], [\e_0, \e_1]] + 3 \cdot [\e_0, [[[[[\e_0, \e_1], \e_1], \e_1], \e_1], \e_1]] \\
    &- \tfrac{65}{16} \cdot [[\e_0, \e_1], [[\e_0, \e_1], [[\e_0, \e_1], \e_1]]] + 4 \cdot [[\e_0, \e_1], [[[[\e_0, \e_1], \e_1], \e_1], \e_1]] \\
    &+ 3 \cdot [[[\e_0, \e_1], \e_1], [[[\e_0, \e_1], \e_1], \e_1]] - [[[[[[\e_0, \e_1], \e_1], \e_1], \e_1], \e_1], \e_1]
.\end{align*}

\end{document}